\documentclass{article}%
\usepackage{pstricks}
\usepackage{pstricks-add}
\usepackage{amsmath}
\usepackage{amsthm}
\usepackage{amsfonts}
\usepackage{amssymb}
\usepackage{fullpage}
\usepackage{graphicx}
\usepackage{color}
\usepackage{enumitem}
\usepackage{varioref}
\usepackage{pgf,tikz}
\usepackage{mathrsfs}
\usepackage{enumitem}
\usepackage{soul}
\usepackage{authblk}%
\allowdisplaybreaks
\newtheorem{theorem}{Theorem}
\newtheorem{lemma}[theorem]{Lemma}

\newtheorem{definition}{Definition}
\newtheorem{notation}[definition]{Notation}
\newtheorem{algorithm}{Algorithm}
\newtheorem{example}{Example}
\newtheorem{remark}{Remark}

\begin{document}

\title{Resultants over Commutative Idempotent Semirings}
\author{Hoon Hong, Yonggu Kim, Georgy Scholten, J. Rafael Sendra 
\thanks{H. Hong was partially supported by US NSF 1319632.
A part of this work was also developed while H. Hong was visiting J.R. Sendra  at
the {\it Universidad de Alcal\'a}, under the frame of the project {\it Giner de los Rios}.
Y. Kim was financially supported by Chonnam National University in the
program, 2012.
J.R. Sendra belongs to  the Research Group ASYNACS (Ref. CCEE2011/R34) and is partially supported by the Spanish "Ministerio de Econom\'{\i}a y Competitividad"  under the Project  MTM2014-54141.}}
\date{}
\maketitle

\begin{abstract}
The resultant plays a crucial role in (computational) algebra and algebraic
geometry. One of the most important and well known properties of the resultant
is that it is equal to the determinant of the Sylvester matrix. In 2008,
Odagiri proved that a similar property holds over the tropical semiring if one
replaces subtraction with addition. The tropical semiring belongs to a large
family of algebraic structures called commutative idempotent semiring. In this
paper, we prove that the same property (with subtraction replaced with
addition) holds over an \emph{arbitrary\/} commutative idempotent semiring.

\end{abstract}

\renewcommand\footnotemark{}



\section{Introduction}

The main contribution of this paper is adapting a certain important property
of resultant (over commutative rings) to commutative idempotent semirings. The
work was inspired by Odagiri's work \cite{Od08} where the property of
resultant is adapted to a particular commutative idempotent semiring, namely
tropical semiring. Below we elaborate on the above statements.

The resultant plays a crucial role in (computational) algebra and algebraic
geometry ~\cite{syl53, Sal85, Lo83, CLO97}. Let
\begin{align*}
f  &  =\left(  x-\alpha_{1}\right)  \cdots\left(  x-\alpha_{m}\right) \\
g  &  =\left(  x-\beta_{1}\right)  \cdots\left(  x-\beta_{n}\right)
\end{align*}
be two polynomials over a commutative ring. The resultant $\mathbf{R}$ of $f$
and $g$ is defined as
\[
\mathbf{R}=\prod_{\substack{1\leq i\leq m\\1\leq j\leq n}}\left(  \alpha
_{i}-\beta_{j}\right)
\]
and the Sylvester expression of $f$ and $g$ is defined as
\[
\mathbf{S}=\det M
\]
where $M$ is a certain matrix whose entries are from the coefficients of the
two polynomials. One of the most important and well known properties of the
resultant~\cite{syl53,DHKS09} is that
\[
\mathbf{R}=\mathbf{S.}
\]

The tropical semiring is a variant of a commutative ring where it is equipped
with a total order and that the addition operation is defined as maximum. As
the result, it does not allow substraction (due to lack of additive inverse;
hence the name semiring). It has been intensively investigated due to numerous
interesting applications ~\cite{Si88, Pi98, PaSt04, HeOlWo06, BoJenSpSt07,
GoMi08, SpSt09, ItMiSh09, But10, StTr13,LitSer14, BrSh14, MacStu15}.

There have been several adaptations of the properties of the resultant (over
commutative rings) to the tropical semiring~\cite{MiGr06, DFS07, BJSST, Ta08,
Od08, JY13}. In particular, Odagiri~\cite{Od08} proved that the property of
the resultant still holds if one simply replaces subtraction with addition,
that is, if we let%
\begin{align*}
f  &  =\left(  x+\alpha_{1}\right)  \cdots\left(  x+\alpha_{m}\right) \\
g  &  =\left(  x+\beta_{1}\right)  \cdots\left(  x+\beta_{n}\right)
\end{align*}
and redefine the resultant as
\[
\mathbf{R}=\prod_{\substack{1\leq i\leq m\\1\leq j\leq n}}\left(  \alpha
_{i}+\beta_{j}\right)  ,
\]
and redefine the Sylvester expression as
\[
\mathbf{S}=\operatorname*{per}M
\]
then
\[
\mathbf{R}=\mathbf{S.}
\]

The tropical semiring belongs to a large family of algebraic structures called
commutative idempotent semiring (CIS)~\cite{Oha70, KM97, GuB98, GoJ99, Gla02}.
As the name indicates, a commutative idempotent semiring is similar to a
commutative ring, except that we do not require substraction (additive
inverse) but we instead require additive idempotency. There are many
interesting algebraic structures that are CIS (see Section~\ref{cisr}).

Hence one naturally wonders whether we can extend Odagiri's result on tropical
semiring to the whole family of CIS. The main contribution of this paper is to
answer affirmatively, proving that Odagiri's result indeed holds for arbitrary
CIS, not just for the tropical semiring.

For proving the property, we, at the beginning, naturally attempted to
generalize/relax the proof technique of Odagiri. However we found it
practically impossible mainly because Odagiri's proof crucially exploits the
fact that the tropical semiring has a total order. Since CIS, in general, does
not require a total order, we had to develop a different proof technique. The
new technique consists of the following four parts:

\begin{enumerate}
\itemsep=-0.2em

\item Represent each term in $\mathbf{R}$ as a certain boolean matrix, which
we call a $res$-representation.

\item Represent each term in $\mathbf{S}$ as a certain pair of boolean
matrices, which we call a $syl$-representation.

\item Show that if a term has a $res$-representation then it has a $syl$-representation.

\item Show that if a term has a $syl$-representation then it has a $res$-representation.
\end{enumerate}

The representation of terms as boolean matrices are not essential from logical
point of view, but they are extremely helpful in discovering, explaining and
understanding the steps of the proof. The proof is constructive, that is, it
provides an algorithm that takes a $res$-representation and produces a
$syl$-representation, and vice versa. The implementation of the algorithms in
Maple~\cite{Mpuser15, Mpprog15} can be downloaded from

\begin{center}
\texttt{http://www.math.ncsu.edu/\textasciitilde hong/rcis/}
\end{center}

Note that all the above discussion is approached and stated algebraically in
the sense that we consider the equality between two algebraic objects
$\mathbf{R}$ and $\mathbf{S}$. Over the commutative ring case (in particular
integral domain), the algebraic equality has an obvious geometric
interpretation, namely two polynomials have common root iff $\mathbf{S}=0$.
Over the tropical semiring case, the algebraic equality can also be given a
similar geometric interpretation, using a modified notion of root~\cite{Od08,
ItMiSh09, BrSh14, MacStu15}. Thus one wonders whether the algebraic equality
over CIS can be also given a similar geometric interpretation. One way for
this would be extending the notion of root further. We leave it as a challenge
for future work.

The paper is structured as follows. In Section~\ref{review}, we recall the
axiomatic definition of CIS and list several algebraic structures that satisfy
the axioms. In Section~\ref{result}, we give a precise statement of the main
result of this paper. In Section~\ref{examples}, we illustrate the main result
on the algebraic structures listed in Section~\ref{review}. In
Section~\ref{sketch}, we informally sketch the overall structure/underlying
ideas of the proof by using a small case. In Section~\ref{proof}, we finally
provide a detailed and general proof of the main result.


\section{Review of Commutative Idempotent Semiring}

\label{review} In this section, we review the definition of commutative
idempotent semiring, and list a few examples. For more details,
see~\cite{KM97, GuJ98,Gla02, Gol03}.

\begin{definition}
[Commutative Idempotent Semiring]\label{cisr} A Commutative Idempotent
Semiring (CIS) is a tuple $(\mathcal{I},+,\times,0,1)$ where $\mathcal{I}$ is
a set, $+$ and $\times$ are binary operations over $\mathcal{I}$ and $0,1$ are
elements of $\mathcal{I}$ such that the following properties hold for all
$a,b,c \in\mathcal{I}$:

\renewcommand{\arraystretch}{1.3}
\begin{tabular}
[c]{|c|c|c|}\hline
$+$ & $\times$ & $+$ and $\times$\\\hline
$a+b=b+a$ & $a\times b=b\times a$ & $\left(  a+b\right)  \times c=$ $a\times
c+b\times c$\\
$\left(  a+b\right)  +c=a+\left(  b+c\right)  $ & $\left(  a\times b\right)
\times c=a\times\left(  b\times c\right)  $ & \\
$a+0=a$ & $a\times1=a$ & \\
$a+ a=a$ & $a\times0=0$ & \ \\\hline
\end{tabular}

\end{definition}

\begin{remark}
Note that CIS does not require the existence of additive inverse, thus
semiring. It instead requires idempotency, thus idempotent semiring.
\end{remark}

\begin{example}
\label{cisr-ex} We list a few examples of commutative idempotent semirings
(CIS). \renewcommand{\arraystretch}{1.4}%

\begin{tabular}
[c]{|l|l|l|l|l|l|}\hline
name & $\mathcal{I}$ & $+$ & $\times$ & $0$ & $1$\\\hline
tropical & $\mathbb{R} \cup\{-\infty\}$ & maximum & addition & $-\infty$ &
$0$\\
power set & power set of a set $S$ & union & intersection & $\emptyset$ &
$S$\\
topology & topology on a set $S$ & union & intersection & $\emptyset$ & $S$\\
compact-convex & compact convex subsets of $\mathbb{R}^{n}$ & convex hull &
Minkowski sum & $\emptyset$ & $\{0\}$\\
sequences & sequences over a CIS $A$ & component-wise & convolution &
$(0_{A},\ldots)$ & $(1_{A},0_{A},\ldots)$\\
ideals & ideals of a comm. ring $R$ with unity & ideal sum & ideal product &
$\{0_{R}\}$ & $\langle1_{R} \rangle$\\\hline
\end{tabular}

\medskip

\noindent For the sequences,
\[
\begin{array}
[t]{llllllll}%
\text{component-wise} & : & \forall\; i \geq0 & (u+v)_{i} & = & u_{i}+_{A}
v_{i} &  & \\
\text{convolution} & : & \forall\;i \geq0 & (u\times v)_{i} & = &
\sum_{A,\;j,k\geq0, j+k=i}\ u_{j}\times_{A}v_{k} &  &
\end{array}
\]

\end{example}


\section{Main Result}

\label{result} In this section, we give a precise statement of the main result
of this paper and, in the next section, we illustrate it on a few examples.
Let $\mathcal{I}$ be a commutative idempotent semiring (CIS). Let $\alpha
_{1},\ldots,\alpha_{m},\beta_{1},\ldots,\beta_{n} \in\mathcal{I}$. Let
\begin{align*}
f  &  =\prod_{i=1}^{m}\left(  x+\alpha_{i}\right)  \;\;\;\in\mathcal{I}[x]\\
g  &  =\prod_{j=1}^{n}\left(  x+\beta_{j}\right)  \;\;\;\in\mathcal{I}[x]
\end{align*}

\begin{definition}
[Resultant]\label{res} The resultant $\mathbf{R}$ of $f$ and $g$ is defined
as
\[
\mathbf{R}=\prod_{\substack{1\leq i\leq m\\1\leq j\leq n}}\left(  \alpha
_{i}+\beta_{j}\right)  \;\;\;\in\mathcal{I.}
\]

\end{definition}

\begin{example}
\label{ex:R} Let $m=3$ and $n=2.$ Then%
\[
\mathbf{R}=\left(  \alpha_{1}+\beta_{1}\right)  \left(  \alpha_{1}+\beta
_{2}\right)  \left(  \alpha_{2}+\beta_{1}\right)  \left(  \alpha_{2}+\beta
_{2}\right)  \left(  \alpha_{3}+\beta_{1}\right)  \left(  \alpha_{3}+\beta
_{2}\right)
\]

\end{example}

\begin{definition}
[Sylvester Matrix]\label{sylmat} Let $a_{0},\ldots,a_{m}\in\mathcal{I} $ and
$b_{0},\ldots,b_{n}\in\mathcal{I} $ be such that
\begin{align*}
f  &  =\sum_{i=0}^{m}a_{m-i}x^{i}\\
g  &  =\sum_{j=0}^{n}b_{n-j}x^{j}%
\end{align*}
Then the Sylvester matrix of $f$and $g$ is defined as%
\[
M=%
\begin{array}
[c]{@{}c@{\hspace{ - 20pt}}c}%
\underbrace{%
\begin{bmatrix}
a_{0} & \cdots & \cdots & \cdots & a_{m} &  & \\
& \ddots &  &  &  & \ddots & \\
&  & a_{0} & \cdots & \cdots & \cdots & a_{m}\\
b_{0} & \cdots & \cdots & b_{n} &  &  & \\
& \ddots &  &  & \ddots &  & \\
&  & \ddots &  &  & \ddots & \\
&  &  & b_{0} & \cdots & \cdots & b_{n}%
\end{bmatrix}
}_{\displaystyle{m+n}} &
\begin{array}
[c]{c}%
\left.
\begin{array}
[c]{c}%
\\
\\
\\
\\
\end{array}
\right\}  n\\
\left.
\begin{array}
[c]{c}%
\\
\\
\\
\\
\\
\end{array}
\right\}  m
\end{array}
\end{array}
\]

\end{definition}

\begin{example}
Let $m=3$ and $n=2.$ Then
\begin{align*}
f  &  =\left(  x+\alpha_{1}\right)  \left(  x+\alpha_{2}\right)  \left(
x+\alpha_{3}\right) \\
&  =a_{0}x^{3}+a_{1}x^{2}+a_{2}x^{1}+a_{3}x^{0}\\
g  &  =\left(  x+\beta_{1}\right)  \left(  x+\beta_{2}\right) \\
&  =b_{0}x^{2}+b_{1}x^{1}+b_{2}x^{0}%
\end{align*}
where%
\begin{align*}
a_{0}  &  =1\\
a_{1}  &  =\alpha_{1}+\alpha_{2}+\alpha_{3}\\
a_{2}  &  =\alpha_{1}\alpha_{2}+\alpha_{1}\alpha_{3}+\alpha_{2}\alpha_{3}\\
a_{3}  &  =\alpha_{1}\alpha_{2}\alpha_{3}\\
b_{0}  &  =1\\
b_{1}  &  =\beta_{1}+\beta_{2}\\
b_{2}  &  =\beta_{1}\beta_{2}%
\end{align*}
Thus the Sylvester matrix $M$ of $f$ and $g$ is given by%
\[
M=
\begin{bmatrix}
1 & \alpha_{1}+\alpha_{2}+\alpha_{3} & \alpha_{1}\alpha_{2}+\alpha_{1}%
\alpha_{3}+\alpha_{2}\alpha_{3} & \alpha_{1}\alpha_{2}\alpha_{3} & \\
& 1 & \alpha_{1}+\alpha_{2}+\alpha_{3} & \alpha_{1}\alpha_{2}+\alpha_{1}%
\alpha_{3}+\alpha_{2}\alpha_{3} & \alpha_{1}\alpha_{2}\alpha_{3}\\
1 & \beta_{1}+\beta_{2} & \beta_{1}\beta_{2} &  & \\
& 1 & \beta_{1}+\beta_{2} & \beta_{1}\beta_{2} & \\
&  & 1 & \beta_{1}+\beta_{2} & \beta_{1}\beta_{2}%
\end{bmatrix}
\]

\end{example}

\begin{definition}
[Sylvester expression]The Sylvester expression $\mathbf{S}$ of $f$ and $g$ is
defined as the permanent of the Sylvester matrix of $f$ and $g$, that is,
\[
\mathbf{S}=\mathop{\rm per}\left(  M\right)  \;\;\;\in\mathcal{I}.
\]

\end{definition}

\begin{example}
\label{ex:S} Let $m=3$ and $n=2.$ Then
\[
\mathbf{S}=\mathop{\rm per}%
\begin{bmatrix}
1 & \alpha_{1}+\alpha_{2}+\alpha_{3} & \alpha_{1}\alpha_{2}+\alpha_{1}%
\alpha_{3}+\alpha_{2}\alpha_{3} & \alpha_{1}\alpha_{2}\alpha_{3} & \\
& 1 & \alpha_{1}+\alpha_{2}+\alpha_{3} & \alpha_{1}\alpha_{2}+\alpha_{1}%
\alpha_{3}+\alpha_{2}\alpha_{3} & \alpha_{1}\alpha_{2}\alpha_{3}\\
1 & \beta_{1}+\beta_{2} & \beta_{1}\beta_{2} &  & \\
& 1 & \beta_{1}+\beta_{2} & \beta_{1}\beta_{2} & \\
&  & 1 & \beta_{1}+\beta_{2} & \beta_{1}\beta_{2}%
\end{bmatrix}
\]

\end{example}

\begin{theorem}
[Main Result]\label{main} $\mathbf{R}=\mathbf{S}.$
\end{theorem}

\begin{example}
We illustrate the ``meaning'' of the main result (Theorem~\ref{main}) by
verifying it on small degrees: $m=3$ and $n=2$ where $\mathbf{R}$ and
$\mathbf{S}$ are described in Example~\ref{ex:R} and~\ref{ex:S}. When we
expand $\mathbf{R}$ and $\mathbf{S}$, we obtain

\medskip

$\mathbf{R}={\alpha_{{1}}}^{2}{\alpha_{{2}}}^{2}{\alpha_{{3}}}^{2}%
+{\alpha_{{1}}}^{2}{\alpha_{{2}}}^{2}\alpha_{{3}}\beta_{{1}}+{\alpha_{{1}}%
}^{2}{\alpha_{{2}}}^{2}\alpha_{{3}}\beta_{{2}}+{\alpha_{{1}}}^{2}{\alpha_{{2}%
}}^{2}\beta_{{1}}\beta_{{2}}+{\alpha_{{1}}}^{2}\alpha_{{2}}{\alpha_{{3}}}%
^{2}\beta_{{1}}+{\alpha_{{1}}}^{2}\alpha_{{2}}{\alpha_{{3}}}^{2}\beta_{{2}%
}+{\alpha_{{1}}}^{2}\alpha_{{2}}\alpha_{{3}}{\beta_{{1}}}^{2}+2\,{\alpha_{{1}%
}}^{2}\alpha_{{2}}\alpha_{{3}}\beta_{{1}}\beta_{{2}}+{\alpha_{{1}}}^{2}%
\alpha_{{2}}\alpha_{{3}}{\beta_{{2}}}^{2}+{\alpha_{{1}}}^{2}\alpha_{{2}}%
{\beta_{{1}}}^{2}\beta_{{2}}+{\alpha_{{1}}}^{2}\alpha_{{2}}\beta_{{1}}%
{\beta_{{2}}}^{2}+{\alpha_{{1}}}^{2}{\alpha_{{3}}}^{2}\beta_{{1}}\beta_{{2}%
}+{\alpha_{{1}}}^{2}\alpha_{{3}}{\beta_{{1}}}^{2}\beta_{{2}}+{\alpha_{{1}}%
}^{2}\alpha_{{3}}\beta_{{1}}{\beta_{{2}}}^{2}+{\alpha_{{1}}}^{2}{\beta_{{1}}%
}^{2}{\beta_{{2}}}^{2}+\alpha_{{1}}{\alpha_{{2}}}^{2}{\alpha_{{3}}}^{2}%
\beta_{{1}}+\alpha_{{1}}{\alpha_{{2}}}^{2}{\alpha_{{3}}}^{2}\beta_{{2}}%
+\alpha_{{1}}{\alpha_{{2}}}^{2}\alpha_{{3}}{\beta_{{1}}}^{2}+2\,\alpha_{{1}%
}{\alpha_{{2}}}^{2}\alpha_{{3}}\beta_{{1}}\beta_{{2}}+\alpha_{{1}}{\alpha
_{{2}}}^{2}\alpha_{{3}}{\beta_{{2}}}^{2}+\alpha_{{1}}{\alpha_{{2}}}^{2}%
{\beta_{{1}}}^{2}\beta_{{2}}+\alpha_{{1}}{\alpha_{{2}}}^{2}\beta_{{1}}%
{\beta_{{2}}}^{2}+\alpha_{{1}}\alpha_{{2}}{\alpha_{{3}}}^{2}{\beta_{{1}}}%
^{2}+2\,\alpha_{{1}}\alpha_{{2}}{\alpha_{{3}}}^{2}\beta_{{1}}\beta_{{2}%
}+\alpha_{{1}}\alpha_{{2}}{\alpha_{{3}}}^{2}{\beta_{{2}}}^{2}+\alpha_{{1}%
}\alpha_{{2}}\alpha_{{3}}{\beta_{{1}}}^{3}+3\,\alpha_{{1}}\alpha_{{2}}%
\alpha_{{3}}{\beta_{{1}}}^{2}\beta_{{2}}+3\,\alpha_{{1}}\alpha_{{2}}%
\alpha_{{3}}\beta_{{1}}{\beta_{{2}}}^{2}+\alpha_{{1}}\alpha_{{2}}\alpha_{{3}%
}{\beta_{{2}}}^{3}+\alpha_{{1}}\alpha_{{2}}{\beta_{{1}}}^{3}\beta_{{2}%
}+2\,\alpha_{{1}}\alpha_{{2}}{\beta_{{1}}}^{2}{\beta_{{2}}}^{2}+\alpha_{{1}%
}\alpha_{{2}}\beta_{{1}}{\beta_{{2}}}^{3}+\alpha_{{1}}{\alpha_{{3}}}^{2}%
{\beta_{{1}}}^{2}\beta_{{2}}+\alpha_{{1}}{\alpha_{{3}}}^{2}\beta_{{1}}%
{\beta_{{2}}}^{2}+\alpha_{{1}}\alpha_{{3}}{\beta_{{1}}}^{3}\beta_{{2}%
}+2\,\alpha_{{1}}\alpha_{{3}}{\beta_{{1}}}^{2}{\beta_{{2}}}^{2}+\alpha_{{1}%
}\alpha_{{3}}\beta_{{1}}{\beta_{{2}}}^{3}+\alpha_{{1}}{\beta_{{1}}}^{3}%
{\beta_{{2}}}^{2}+\alpha_{{1}}{\beta_{{1}}}^{2}{\beta_{{2}}}^{3}+{\alpha_{{2}%
}}^{2}{\alpha_{{3}}}^{2}\beta_{{1}}\beta_{{2}}+{\alpha_{{2}}}^{2}\alpha_{{3}%
}{\beta_{{1}}}^{2}\beta_{{2}}+{\alpha_{{2}}}^{2}\alpha_{{3}}\beta_{{1}}%
{\beta_{{2}}}^{2}+{\alpha_{{2}}}^{2}{\beta_{{1}}}^{2}{\beta_{{2}}}^{2}%
+\alpha_{{2}}{\alpha_{{3}}}^{2}{\beta_{{1}}}^{2}\beta_{{2}}+\alpha_{{2}%
}{\alpha_{{3}}}^{2}\beta_{{1}}{\beta_{{2}}}^{2}+\alpha_{{2}}\alpha_{{3}}%
{\beta_{{1}}}^{3}\beta_{{2}}+2\,\alpha_{{2}}\alpha_{{3}}{\beta_{{1}}}%
^{2}{\beta_{{2}}}^{2}+\alpha_{{2}}\alpha_{{3}}\beta_{{1}}{\beta_{{2}}}%
^{3}+\alpha_{{2}}{\beta_{{1}}}^{3}{\beta_{{2}}}^{2}+\alpha_{{2}}{\beta_{{1}}%
}^{2}{\beta_{{2}}}^{3}+{\alpha_{{3}}}^{2}{\beta_{{1}}}^{2}{\beta_{{2}}}%
^{2}+\alpha_{{3}}{\beta_{{1}}}^{3}{\beta_{{2}}}^{2}+\alpha_{{3}}{\beta_{{1}}%
}^{2}{\beta_{{2}}}^{3}+{\beta_{{1}}}^{3}{\beta_{{2}}}^{3}$

\medskip

$\mathbf{S}={\alpha_{{1}}}^{2}{\alpha_{{2}}}^{2}{\alpha_{{3}}}^{2}%
+{\alpha_{{1}}}^{2}{\alpha_{{2}}}^{2}\alpha_{{3}}\beta_{{1}}+{\alpha_{{1}}%
}^{2}{\alpha_{{2}}}^{2}\alpha_{{3}}\beta_{{2}}+{\alpha_{{1}}}^{2}{\alpha_{{2}%
}}^{2}\beta_{{1}}\beta_{{2}}+{\alpha_{{1}}}^{2}\alpha_{{2}}{\alpha_{{3}}}%
^{2}\beta_{{1}}+{\alpha_{{1}}}^{2}\alpha_{{2}}{\alpha_{{3}}}^{2}\beta_{{2}%
}+{\alpha_{{1}}}^{2}\alpha_{{2}}\alpha_{{3}}{\beta_{{1}}}^{2}+6\,{\alpha_{{1}%
}}^{2}\alpha_{{2}}\alpha_{{3}}\beta_{{1}}\beta_{{2}}+{\alpha_{{1}}}^{2}%
\alpha_{{2}}\alpha_{{3}}{\beta_{{2}}}^{2}+{\alpha_{{1}}}^{2}\alpha_{{2}}%
{\beta_{{1}}}^{2}\beta_{{2}}+{\alpha_{{1}}}^{2}\alpha_{{2}}\beta_{{1}}%
{\beta_{{2}}}^{2}+{\alpha_{{1}}}^{2}{\alpha_{{3}}}^{2}\beta_{{1}}\beta_{{2}%
}+{\alpha_{{1}}}^{2}\alpha_{{3}}{\beta_{{1}}}^{2}\beta_{{2}}+{\alpha_{{1}}%
}^{2}\alpha_{{3}}\beta_{{1}}{\beta_{{2}}}^{2}+{\alpha_{{1}}}^{2}{\beta_{{1}}%
}^{2}{\beta_{{2}}}^{2}+\alpha_{{1}}{\alpha_{{2}}}^{2}{\alpha_{{3}}}^{2}%
\beta_{{1}}+\alpha_{{1}}{\alpha_{{2}}}^{2}{\alpha_{{3}}}^{2}\beta_{{2}}%
+\alpha_{{1}}{\alpha_{{2}}}^{2}\alpha_{{3}}{\beta_{{1}}}^{2}+6\,\alpha_{{1}%
}{\alpha_{{2}}}^{2}\alpha_{{3}}\beta_{{1}}\beta_{{2}}+\alpha_{{1}}{\alpha
_{{2}}}^{2}\alpha_{{3}}{\beta_{{2}}}^{2}+\alpha_{{1}}{\alpha_{{2}}}^{2}%
{\beta_{{1}}}^{2}\beta_{{2}}+\alpha_{{1}}{\alpha_{{2}}}^{2}\beta_{{1}}%
{\beta_{{2}}}^{2}+\alpha_{{1}}\alpha_{{2}}{\alpha_{{3}}}^{2}{\beta_{{1}}}%
^{2}+6\,\alpha_{{1}}\alpha_{{2}}{\alpha_{{3}}}^{2}\beta_{{1}}\beta_{{2}%
}+\alpha_{{1}}\alpha_{{2}}{\alpha_{{3}}}^{2}{\beta_{{2}}}^{2}+\alpha_{{1}%
}\alpha_{{2}}\alpha_{{3}}{\beta_{{1}}}^{3}+9\,\alpha_{{1}}\alpha_{{2}}%
\alpha_{{3}}{\beta_{{1}}}^{2}\beta_{{2}}+9\,\alpha_{{1}}\alpha_{{2}}%
\alpha_{{3}}\beta_{{1}}{\beta_{{2}}}^{2}+\alpha_{{1}}\alpha_{{2}}\alpha_{{3}%
}{\beta_{{2}}}^{3}+\alpha_{{1}}\alpha_{{2}}{\beta_{{1}}}^{3}\beta_{{2}%
}+6\,\alpha_{{1}}\alpha_{{2}}{\beta_{{1}}}^{2}{\beta_{{2}}}^{2}+\alpha_{{1}%
}\alpha_{{2}}\beta_{{1}}{\beta_{{2}}}^{3}+\alpha_{{1}}{\alpha_{{3}}}^{2}%
{\beta_{{1}}}^{2}\beta_{{2}}+\alpha_{{1}}{\alpha_{{3}}}^{2}\beta_{{1}}%
{\beta_{{2}}}^{2}+\alpha_{{1}}\alpha_{{3}}{\beta_{{1}}}^{3}\beta_{{2}%
}+6\,\alpha_{{1}}\alpha_{{3}}{\beta_{{1}}}^{2}{\beta_{{2}}}^{2}+\alpha_{{1}%
}\alpha_{{3}}\beta_{{1}}{\beta_{{2}}}^{3}+\alpha_{{1}}{\beta_{{1}}}^{3}%
{\beta_{{2}}}^{2}+\alpha_{{1}}{\beta_{{1}}}^{2}{\beta_{{2}}}^{3}+{\alpha_{{2}%
}}^{2}{\alpha_{{3}}}^{2}\beta_{{1}}\beta_{{2}}+{\alpha_{{2}}}^{2}\alpha_{{3}%
}{\beta_{{1}}}^{2}\beta_{{2}}+{\alpha_{{2}}}^{2}\alpha_{{3}}\beta_{{1}}%
{\beta_{{2}}}^{2}+{\alpha_{{2}}}^{2}{\beta_{{1}}}^{2}{\beta_{{2}}}^{2}%
+\alpha_{{2}}{\alpha_{{3}}}^{2}{\beta_{{1}}}^{2}\beta_{{2}}+\alpha_{{2}%
}{\alpha_{{3}}}^{2}\beta_{{1}}{\beta_{{2}}}^{2}+\alpha_{{2}}\alpha_{{3}}%
{\beta_{{1}}}^{3}\beta_{{2}}+6\,\alpha_{{2}}\alpha_{{3}}{\beta_{{1}}}%
^{2}{\beta_{{2}}}^{2}+\alpha_{{2}}\alpha_{{3}}\beta_{{1}}{\beta_{{2}}}%
^{3}+\alpha_{{2}}{\beta_{{1}}}^{3}{\beta_{{2}}}^{2}+\alpha_{{2}}{\beta_{{1}}%
}^{2}{\beta_{{2}}}^{3}+{\alpha_{{3}}}^{2}{\beta_{{1}}}^{2}{\beta_{{2}}}%
^{2}+\alpha_{{3}}{\beta_{{1}}}^{3}{\beta_{{2}}}^{2}+\alpha_{{3}}{\beta_{{1}}%
}^{2}{\beta_{{2}}}^{3}+{\beta_{{1}}}^{3}{\beta_{{2}}}^{3}$

\medskip Observe that $\mathbf{R}$ and $\mathbf{S}$ have the same terms. Some
terms appear a different number of times. For example, ${\alpha_{{1}}}%
^{2}\alpha_{{2}}\alpha_{{3}}\beta_{{1}}\beta_{{2}}$ appears two times in
$\mathbf{R}$ and six times in $\mathbf{S}.$ However, since an commutative
idempotent semiring ignores additive multiplicities, it does not matter. Hence
we see that $\mathbf{R} =\mathbf{S}$.
\end{example}


\section{Examples}

\label{examples} In this section, we will show computational examples on
\emph{particular} CIS structures given in Example~\ref{cisr-ex} of
Section~\ref{review} to confirm the validity of the main result
(Theorem~\ref{main}) before its general proof (given in Section~\ref{proof}).
We will use \emph{structure-specific languages} whenever possible, in the hope
that it would demonstrate the applicability of the main result in apparently
very different contexts. We will confirm the main result via direct
structure-specific computations. For easy computation, we consider only degree
two polynomials.

\begin{example}
[Tropical]Let $\mathcal{I}$ be the CIS of tropical semiring. Consider
\renewcommand\arraystretch{1.2}
\[
f=\max\left(  x,1\right)  +\max\left(  x,3\right)  ,\ \ g=\max\left(
x,2\right)  +\max\left(  x,4\right)
\]
We show that $\mathbf{R=S}$ for the above $f$ and $g$, via direct
computations. Note
\[
\begin{array}
[c]{lll}%
\mathbf{R} & = & \max\left(  1,2\right)  +\max\left(  1,4\right)  +\max\left(
3,2\right)  +\max\left(  3,4\right) \\
& = & 2+4+3+4\\
& = & 13\\
&  & \\
\mathbf{S} & = & \mathrm{per}\left[
\begin{array}
[c]{cccc}%
0 & \max\left(  1,3\right)  & 1+3 & -\infty\\
-\infty & 0 & \max\left(  1,3\right)  & \max\left(  1,3\right) \\
0 & \max\left(  2,4\right)  & 2+4 & -\infty\\
-\infty & 0 & \max\left(  2,4\right)  & 2+4
\end{array}
\right] \\
& = & \mathrm{per}\left[
\begin{array}
[c]{cccc}%
0 & 3 & 4 & -\infty\\
-\infty & 0 & 3 & 4\\
0 & 4 & 6 & -\infty\\
-\infty & 0 & 4 & 6
\end{array}
\right] \\
& = & \max\left(  4+6,3+4+4,3+4+6,2\cdot4,2\cdot6,4+2\cdot4,2\cdot3+6\right)
\\
& = & \max\left(  10,11,13,8,12,12,12\right) \\
& = & 13
\end{array}
\]
Thus we have $\mathbf{R}=\mathbf{S}$.
\end{example}

\begin{example}
[Power set]Let $\mathcal{I}$ be the CIS of the power set of~$\mathbb{R}$.
Consider \renewcommand\arraystretch{1.2}
\[
f=(x\cup[1,2])\;\cap\;(x\cup[3,4]),\;\;\;\; g=\left(  x\cup[2,3]\right)
\;\cap\;\left(  x\cup[4,5]\right)
\]
We show that $\mathbf{R=S}$ for the above $f$ and $g$, via direct
computations. Note
\[%
\begin{array}
[c]{lll}%
\mathbf{R} & = & \left(  [1,2] \cup[2,3]\right)  \;\;\cap\;\; \left(  [1,2]
\cup[4,5]\right)  \;\;\cap\;\; \left(  [3,4] \cup[2,3]\right)  \;\;\cap\;\;
\left(  [3,4] \cup[4,5]\right) \\
& = & [1,3]\;\;\cap\;\;\left(  [1,2]\cup[4,5]\right)  \;\;\cap\;\; [2,4]
\;\;\cap\;\; [3,5]\\
& = & [1,3]\;\;\cap\;\;\left(  [1,2]\cup[4,5]\right)  \;\;\cap\;\; [3,4]\\
& = & [3,3]\;\;\cap\;\;\left(  [1,2]\cup[4,5]\right) \\
& = & \emptyset\\
&  & \\
\mathbf{S} & = & \mathrm{per}\left[
\begin{array}
[c]{cccc}%
\mathbb{R} & [1,2]\cup[3,4] & [1,2]\cap[3,4] & \emptyset\\
\emptyset & \mathbb{R} & [1,2]\cup[3,4] & [1,2]\cap[3,4]\\
\mathbb{R} & [2,3]\cup[4,5] & [2,3]\cap[4,5] & \emptyset\\
\emptyset & \mathbb{R} & [2,3]\cup[4,5] & [2,3]\cap[4,5]
\end{array}
\right] \\
& = & \mathrm{per}\left[
\begin{array}
[c]{cccc}%
\mathbb{R} & [1,2]\cup[3,4] & \emptyset & \emptyset\\
\emptyset & \mathbb{R} & [1,2]\cup[3,4] & \emptyset\\
\mathbb{R} & [2,3]\cup[4,5] & \emptyset & \emptyset\\
\emptyset & \mathbb{R} & [2,3]\cup[4,5] & \emptyset
\end{array}
\right] \\
& = & \emptyset
\end{array}
\]
Thus we have $\mathbf{R}=\mathbf{S}$.
\end{example}

\begin{example}
[Topology]Let $\mathcal{I}$ be a CIS of topology. In particular, we consider
the cofinite topology on~$\mathbb{R}$; that is, open subsets are the empty set
and the complementary of finite subsets of $\mathbb{R}$. This CIS is, of
course, a sub-CIS of the power set CIS of~$\mathbb{R}$ in the previous
example. But it is still interesting to consider, due to its importance.
Consider
\[
f=\Big(x\cup(\mathbb{R} \setminus\{0,1\})\Big)\;\cap\;\Big(x\cup(\mathbb{R}
\setminus\{0,2\})\Big),\;\;\;\; g=\Big(x\cup(\mathbb{R} \setminus
\{0,-1\})\Big)\;\cap\;\Big(  x\cup(\mathbb{R} \setminus\{0,-2\})\Big)
\]
We show that $\mathbf{R=S}$ for the above $f$ and $g$, via direct
computations. Note \renewcommand\arraystretch{1.3}
\[%
\begin{array}
[c]{lll}%
\mathbf{R} & = & \left(  (\mathbb{R} \setminus\{0,1\}) \cup(\mathbb{R}
\setminus\{0,-1\}) \right)  \;\;\cap\;\;\left(  (\mathbb{R} \setminus\{0,1\})
\cup(\mathbb{R} \setminus\{0,-2\}) \right)  \;\;\cap\;\;\\
&  & \left(  (\mathbb{R} \setminus\{0,2\}) \cup(\mathbb{R} \setminus
\{0,-1\})\right)  \;\;\cap\;\; \left(  (\mathbb{R} \setminus\{0,2\})
\cup(\mathbb{R} \setminus\{0,-2\}) \right) \\
& = & \mathbb{R}\setminus\{0\}\\
&  & \\
\mathbf{S} & = & \mathrm{per}\left[
\begin{array}
[c]{cccc}%
\mathbb{R} & (\mathbb{R} \setminus\{0,1\})\cup(\mathbb{R} \setminus\{0,2\}) &
(\mathbb{R} \setminus\{0,1\})\cap(\mathbb{R} \setminus\{0,2\}) & \emptyset\\
\emptyset & \mathbb{R} & (\mathbb{R} \setminus\{0,1\})\cup(\mathbb{R}
\setminus\{0,2\}) & (\mathbb{R} \setminus\{0,1\})\cap(\mathbb{R}
\setminus\{0,2\})\\
\mathbb{R} & (\mathbb{R}\setminus\{0,-1\})\cup(\mathbb{R} \setminus\{0,-2\}) &
(\mathbb{R} \setminus\{0,-1\})\cap(\mathbb{R} \setminus\{0,-2\}) & \emptyset\\
\emptyset & \mathbb{R} & (\mathbb{R}\setminus\{0,-1\})\cup(\mathbb{R}
\setminus\{0,-2\}) & (\mathbb{R}\setminus\{0,-1\})\cap(\mathbb{R}
\setminus\{0,-2\})
\end{array}
\right] \\
& = & \mathrm{per}\left[
\begin{array}
[c]{cccc}%
\mathbb{R} & \mathbb{R} \setminus\{0\} & \mathbb{R} \setminus\{0,1,2\} &
\emptyset\\
\emptyset & \mathbb{R} & \mathbb{R} \setminus\{0\} & \mathbb{R} \setminus
\{0,1,2\}\\
\mathbb{R} & \mathbb{R}\setminus\{0\} & \mathbb{R} \setminus\{0,-1,-2\} &
\emptyset\\
\emptyset & \mathbb{R} & \mathbb{R}\setminus\{0\} & \mathbb{R} \setminus
\{0,-1,-2\}
\end{array}
\right] \\
& = & (\mathbb{R} \setminus\{0\})^{2} \cup((\mathbb{R} \setminus\{0\})^{2}%
\cap(\mathbb{R} \setminus\{0,-1,-2\}) \cup((\mathbb{R} \setminus\{0\})^{2}%
\cap(\mathbb{R} \setminus\{0,1,2\})) \cup\\
&  & (\mathbb{R} \setminus\{0,-1,-2\})^{2} \cup((\mathbb{R} \setminus
\{0,-1,-2\})^{2})\cap(\mathbb{R} \setminus\{0,1,2\})) \cup(\mathbb{R}
\setminus\{0,1,2\})^{2}\\
& = & (\mathbb{R} \setminus\{0\}) \cup(\mathbb{R} \setminus\{0,-1,-2\})
\cup(\mathbb{R} \setminus\{0,1,2\}) \cup(\mathbb{R} \setminus\{0,-1,-2\})
\cup((\mathbb{R} \setminus\{0,-1,-2,1,2\}) \cup(\mathbb{R} \setminus
\{0,1,2\})\\
& = & \mathbb{R} \setminus\{0\}
\end{array}
\]
Thus we have $\mathbf{R}=\mathbf{S}$.
\end{example}

\begin{example}
[Compact-convex]Let $\mathcal{I}$ be the CIS of the set of all compact convex
subsets of $\mathbb{R}^{2}$. Let $\langle x_{1},\ldots,x_{r} \rangle$ denote
the convex hull of the points $x_{1},\ldots,x_{r} \in\mathbb{R}^{2}$.
Consider
\[
f=\Big(x+\langle(0,0)\rangle\Big)\cdot\Big(x+\langle(0,0),(1,0)\rangle
\Big),\,\;\;\;g=\Big(x+\langle(0,0)\rangle\Big)\cdot\Big(x+\langle
(0,0),(0,1)\rangle\Big).
\]
We show that $\mathbf{R=S}$ for the above $f$ and $g$, via direct
computations. Note \renewcommand\arraystretch{1.3}
\[%
\begin{array}
[c]{lll}%
\mathbf{R} & = & \Big(\langle(0,0)\rangle+\langle(0,0)\rangle\Big)\cdot
\Big(\langle(0,0)\rangle+\langle(0,0),(0,1)\rangle\Big)\cdot\Big (\langle
(0,0),(1,0)\rangle+\langle(0,0)\rangle\Big)\cdot\Big(\langle(0,0),(1,0)\rangle
+\langle(0,0),(0,1)\rangle\Big )\\
& = & \left\langle (0,0)\right\rangle \cdot\left\langle
(0,0),(0,1)\right\rangle \cdot\left\langle (0,0),(1,0)\right\rangle
\cdot\left\langle (0,0),(0,1),(1,0)\right\rangle \ \ \text{}\\
& = & \left\langle (0,0),(0,2),(1,2),(2,0),(2,1)\right\rangle \ \text{}\\
&  & \\
\mathbf{S} & = & \mathrm{per}\left[
\begin{array}
[c]{cccc}%
\langle(0,0)\rangle & \langle(0,0)\rangle+\langle(0,0),(1,0)\rangle &
\langle(0,0)\rangle\cdot\langle(0,0),(1,0)\rangle & \emptyset\\
\emptyset & \langle(0,0)\rangle & \langle(0,0)\rangle+\langle
(0,0),(1,0)\rangle & \langle(0,0)\rangle\cdot\langle(0,0),(1,0)\rangle\\
\langle(0,0)\rangle & \langle(0,0)\rangle+\langle(0,0),(0,1)\rangle &
\langle(0,0)\rangle\cdot\langle(0,0),(0,1)\rangle & \emptyset\\
\emptyset & \langle(0,0)\rangle & \langle(0,0)\rangle+\langle
(0,0),(0,1)\rangle & \langle(0,0)\rangle\cdot\langle(0,0),(0,1)\rangle
\end{array}
\right] \\
& = & \mathrm{per}\left[
\begin{array}
[c]{cccc}%
\langle(0,0)\rangle & \langle(0,0),(1,0)\rangle & \langle(0,0),(1,0)\rangle &
\emptyset\\
\emptyset & \langle(0,0)\rangle & \langle(0,0),(1,0)\rangle & \langle
(0,0),(1,0)\rangle\\
\langle(0,0)\rangle & \langle(0,0),(0,1)\rangle & \langle(0,0),(0,1)\rangle &
\emptyset\\
\emptyset & \langle(0,0)\rangle & \langle(0,0),(0,1)\rangle & \langle
(0,0),(0,1)\rangle
\end{array}
\right] \\
& = & \langle(0,0),(1,0)\rangle\cdot\langle(0,0),(0,1)\rangle+\langle
(0,0),(1,0)\rangle^{2}+\langle(0,0),(0,1)\rangle^{2}+\\
&  & \langle(0,0),(1,0)\rangle\cdot\langle(0,0),(0,1)\rangle^{2}+
\langle(0,0),(1,0)\rangle^{2}\cdot\langle(0,0),(0,1)\rangle\\
& = & \left\langle (0,0),(0,1),(1,0),(1,1)\right\rangle +\left\langle
(0,0),(2,0)\right\rangle +\left\langle (0,0),(0,1),(0,2)\right\rangle +\\
&  & \left\langle (0,0),(0,2),(1,0),(1,2)\right\rangle +\left\langle
(0,0),(0,1),(2,0),(2,1)\right\rangle \ \ \text{}\\
& = & \left\langle (0,0),(0,2),(1,2),(2,0),(2,1)\right\rangle \ \ \ \text{}%
\end{array}
\]
Thus we have $\mathbf{R}=\mathbf{S}$.
\end{example}

\begin{example}
[Sequences]$^{{}}$Let $\mathcal{I}$ be the CIS of the set of all sequences
over the boolean algebra \{\texttt{T},\texttt{F}\}. Consider
\[
f = \Big(x+(\mathtt{T},\mathtt{T},\mathtt{F},\ldots)\Big)\Big(x+(\mathtt{T}%
,\mathtt{T},\mathtt{T},\mathtt{F},\ldots)\Big),\;\;\; g = \Big(x+(\mathtt{T}%
,\mathtt{T},\mathtt{T},\mathtt{T},\mathtt{F},\ldots)\Big)\Big(x+(\mathtt{T}%
,\mathtt{T},\mathtt{T},\mathtt{T},\mathtt{T},\mathtt{F},\ldots)\Big)
\]
In order to simplify the presentation, we will use the following short-hands.
Let $s_{i}$ denote the sequence $(\mathtt{T},\ldots,\mathtt{T},\mathtt{F}%
,\ldots)$ where \texttt{T} appears $i+1$ times. Then we can write $f$ and $g$
succinctly as
\[
f=(x+s_{1})(x+s_{2}),\,\;\;\;g=(x+s_{3})(x+s_{4}).
\]
Note that $s_{-1}$ is the additive identity and $s_{0}$ is the multiplicative
identity. Note also that $s_{i}+s_{j}=s_{\max\left\{  i,j\right\}  }$ and
$s_{i}s_{j}=s_{i+j}$ for $i,j\geq0.$ \medskip

\noindent We show that $\mathbf{R=S}$ for the above $f$ and $g$, via direct
computations. Note

\noindent\renewcommand\arraystretch{1.3}
\[%
\begin{array}
[c]{lll}%
\mathbf{R} & = & (s_{1}+s_{3})(s_{1}+s_{4})(s_{2}+s_{3})(s_{2}+s_{4})\\
& = & s_{3}s_{4}s_{3}s_{4}\\
& = & s_{14}\\
& = & (\mathtt{T},\ldots\mathtt{T},\mathtt{F},\ldots) \;\;\;\text{where
\texttt{T} is repeated 15 times}\\
&  & \\
\mathbf{S} & = & \mathrm{per}\left[
\begin{array}
[c]{cccc}%
s_{0} & s_{1}+s_{2} & s_{1}s_{2} & s_{-1}\\
s_{-1} & s_{0} & s_{1}+s_{2} & s_{1}s_{2}\\
s_{0} & s_{3}+s_{4} & s_{3}s_{4} & s_{-1}\\
s_{-1} & s_{0} & s_{3}+s_{4} & s_{3}s_{4}%
\end{array}
\right] \\
& = & \mathrm{per}\left[
\begin{array}
[c]{cccc}%
s_{0} & s_{2} & s_{3} & s_{-1}\\
s_{-1} & s_{0} & s_{2} & s_{3}\\
s_{0} & s_{4} & s_{7} & s_{-1}\\
s_{-1} & s_{0} & s_{4} & s_{7}%
\end{array}
\right] \\
& = & s_{2}^{2}s_{7}+s_{2}s_{3}s_{4}+s_{2}s_{4}s_{7}+s_{3}s_{4}^{2}+s_{3}%
^{2}+s_{3}s_{7}+s_{7}^{2}\\
& = & s_{11}+s_{9}+s_{13}+s_{11}+s_{6}+s_{10}+s_{14}\\
& = & s_{14}\\
& = & (\mathtt{T},\ldots\mathtt{T},\mathtt{F},\ldots) \;\;\;\text{where
\texttt{T} is repeated 15 times}%
\end{array}
\]
Thus we have $\mathbf{R}=\mathbf{S}$.
\end{example}

\begin{example}
[Ideals]Let $\mathcal{I}$ be the CIS of the set of all the ideals of
$\mathbb{C}[v_{1},v_{2}]$. Consider
\[
f=\Big(x+\langle v_{1}^{2}+v_{2}^{2}-1^{2}\rangle\Big)\Big(x+\langle v_{1}%
^{2}+v_{2}^{2}-2^{2}\rangle\Big),\,\;\;\; g=\Big(x+\langle v_{1}^{2}-v_{2}%
^{2}-1^{2}\rangle\Big)\Big(x+\langle v_{1}^{2}-v_{2}^{2}-2^{2}\rangle\Big)
\]
In order to simplify the presentation, we will use the following short-hands.
\[
I_{i}=\langle v_{1}^{2}+v_{2}^{2}-i^{2}\rangle,\;\;\;J_{j}=\langle v_{1}%
^{2}-v_{2}^{2}-j^{2}\rangle
\]
Then we can write $f$ and $g$ succinctly as
\[
f=(x+I_{1})(x+I_{2}),\,\;\;\;g=(x+J_{1})(x+J_{2})
\]
We show that $\mathbf{R=S}$ for the particular $f$ and $g$, via direct
computations. Note \renewcommand\arraystretch{1.2}
\[
\begin{array}
[c]{lll}%
\mathbf{R} & = & \left(  I_{1}+J_{1}\right)  \left(  I_{1}+J_{2}\right)
\left(  I_{2}+J_{1}\right)  \left(  I_{2}+J_{2}\right) \\
&  & \\
\mathbf{S} & = & \mathrm{per}\left[
\begin{array}
[c]{cccc}%
\langle1 \rangle & I_{1}+I_{2} & I_{1}I_{2} & \left\{  0\right\} \\
\left\{  0\right\}  & \langle1 \rangle & I_{1}+I_{2} & I_{1}I_{2}\\
\langle1 \rangle & J_{1}+J_{2} & J_{1}J_{2} & \left\{  0\right\} \\
\left\{  0\right\}  & \langle1 \rangle & J_{1}+J_{2} & J_{1}J_{2}%
\end{array}
\right] \\
& = & \mathrm{per}\left[
\begin{array}
[c]{cccc}%
\langle1 \rangle & \langle1 \rangle & I_{1}I_{2} & \left\{  0\right\} \\
\left\{  0\right\}  & \langle1 \rangle & \langle1 \rangle & I_{1}I_{2}\\
\langle1 \rangle & \langle1 \rangle & J_{1}J_{2} & \left\{  0\right\} \\
\left\{  0\right\}  & \langle1 \rangle & \langle1 \rangle & J_{1}J_{2}%
\end{array}
\right] \\
& = & (I_{1}I_{2})^{2} + (J_{1}J_{2})^{2} + (I_{1}I_{2})(J_{1}J_{2}) +
I_{1}I_{2} + J_{1} J_{2}%
\end{array}
\]

\noindent After carrying out ideal additions and multiplications in a
straightforward manner~\cite{AtMc69, CLO97, Ku13}, we obtain \medskip

$\mathbf{R}=\langle{v_{{1}}}^{8}-2\,{v_{{1}}}^{4}{v_{{2}}}^{4}+{v_{{2}}}%
^{8}-10 \,{v_{{1}}}^{6}+10\,{v_{{1}}}^{2}{v_{{2}}}^{4}+33\,{v_{{1}}}^{4}-17\,{
v_{{2}}}^{4}-40\,{v_{{1}}}^{2}+16,{v_{{1}}}^{8}-2\,{v_{{1}}}^{4}{v_{{2 }}}%
^{4}+{v_{{2}}}^{8}-13\,{v_{{1}}}^{6}-3\,{v_{{1}}}^{4}{v_{{2}}}^{2}+
13\,{v_{{1}}}^{2}{v_{{2}}}^{4}+3\,{v_{{2}}}^{6}+60\,{v_{{1}}}^{4}+24\,
{v_{{1}}}^{2}{v_{{2}}}^{2}-20\,{v_{{2}}}^{4}-112\,{v_{{1}}}^{2}-48\,{v _{{2}}%
}^{2}+64,{v_{{1}}}^{8}-2\,{v_{{1}}}^{4}{v_{{2}}}^{4}+{v_{{2}}}^{
8}-13\,{v_{{1}}}^{6}+3\,{v_{{1}}}^{4}{v_{{2}}}^{2}+13\,{v_{{1}}}^{2}{v _{{2}}%
}^{4}-3\,{v_{{2}}}^{6}+60\,{v_{{1}}}^{4}-24\,{v_{{1}}}^{2}{v_{{2 }}}%
^{2}-20\,{v_{{2}}}^{4}-112\,{v_{{1}}}^{2}+48\,{v_{{2}}}^{2}+64,{v_{ {1}}}%
^{8}-2\,{v_{{1}}}^{4}{v_{{2}}}^{4}+{v_{{2}}}^{8}-7\,{v_{{1}}}^{6} -3\,{v_{{1}%
}}^{4}{v_{{2}}}^{2}+7\,{v_{{1}}}^{2}{v_{{2}}}^{4}+3\,{v_{{2 }}}^{6}%
+15\,{v_{{1}}}^{4}+6\,{v_{{1}}}^{2}{v_{{2}}}^{2}-5\,{v_{{2}}}^{ 4}%
-13\,{v_{{1}}}^{2}-3\,{v_{{2}}}^{2}+4,{v_{{1}}}^{8}-2\,{v_{{1}}}^{4} {v_{{2}}%
}^{4}+{v_{{2}}}^{8}-7\,{v_{{1}}}^{6}+3\,{v_{{1}}}^{4}{v_{{2}}} ^{2}%
+7\,{v_{{1}}}^{2}{v_{{2}}}^{4}-3\,{v_{{2}}}^{6}+15\,{v_{{1}}}^{4}- 6\,{v_{{1}%
}}^{2}{v_{{2}}}^{2}-5\,{v_{{2}}}^{4}-13\,{v_{{1}}}^{2}+3\,{v _{{2}}}%
^{2}+4,{v_{{1}}}^{8}-2\,{v_{{1}}}^{6}{v_{{2}}}^{2}+2\,{v_{{1}}} ^{2}{v_{{2}}%
}^{6}-{v_{{2}}}^{8}-13\,{v_{{1}}}^{6}+21\,{v_{{1}}}^{4}{v_{{2}}}%
^{2}-3\,{v_{{1}}}^{2}{v_{{2}}}^{4}-5\,{v_{{2}}}^{6}+60\,{v_{{1}} }%
^{4}-72\,{v_{{1}}}^{2}{v_{{2}}}^{2}+12\,{v_{{2}}}^{4}-112\,{v_{{1}}}%
^{2}+80\,{v_{{2}}}^{2}+64,{v_{{1}}}^{8}-2\,{v_{{1}}}^{6}{v_{{2}}}^{2}+2
\,{v_{{1}}}^{2}{v_{{2}}}^{6}-{v_{{2}}}^{8}-10\,{v_{{1}}}^{6}+12\,{v_{{ 1}}%
}^{4}{v_{{2}}}^{2}+6\,{v_{{1}}}^{2}{v_{{2}}}^{4}-8\,{v_{{2}}}^{6}+
33\,{v_{{1}}}^{4}-18\,{v_{{1}}}^{2}{v_{{2}}}^{2}-15\,{v_{{2}}}^{4}-40
\,{v_{{1}}}^{2}+8\,{v_{{2}}}^{2}+16,{v_{{1}}}^{8}-2\,{v_{{1}}}^{6}{v_{ {2}}%
}^{2}+2\,{v_{{1}}}^{2}{v_{{2}}}^{6}-{v_{{2}}}^{8}-10\,{v_{{1}}}^{6
}+18\,{v_{{1}}}^{4}{v_{{2}}}^{2}-6\,{v_{{1}}}^{2}{v_{{2}}}^{4}-2\,{v_{ {2}}%
}^{6}+33\,{v_{{1}}}^{4}-48\,{v_{{1}}}^{2}{v_{{2}}}^{2}+15\,{v_{{2} }}%
^{4}-40\,{v_{{1}}}^{2}+32\,{v_{{2}}}^{2}+16,{v_{{1}}}^{8}-2\,{v_{{1} }}%
^{6}{v_{{2}}}^{2}+2\,{v_{{1}}}^{2}{v_{{2}}}^{6}-{v_{{2}}}^{8}-7\,{v_{{1}}}%
^{6}+9\,{v_{{1}}}^{4}{v_{{2}}}^{2}+3\,{v_{{1}}}^{2}{v_{{2}}}^{4} -5\,{v_{{2}}%
}^{6}+15\,{v_{{1}}}^{4}-12\,{v_{{1}}}^{2}{v_{{2}}}^{2}-3\, {v_{{2}}}%
^{4}-13\,{v_{{1}}}^{2}+5\,{v_{{2}}}^{2}+4,{v_{{1}}}^{8}+2\,{v _{{1}}}%
^{6}{v_{{2}}}^{2}-2\,{v_{{1}}}^{2}{v_{{2}}}^{6}-{v_{{2}}}^{8}- 13\,{v_{{1}}%
}^{6}-21\,{v_{{1}}}^{4}{v_{{2}}}^{2}-3\,{v_{{1}}}^{2}{v_{{ 2}}}^{4}%
+5\,{v_{{2}}}^{6}+60\,{v_{{1}}}^{4}+72\,{v_{{1}}}^{2}{v_{{2}}} ^{2}%
+12\,{v_{{2}}}^{4}-112\,{v_{{1}}}^{2}-80\,{v_{{2}}}^{2}+64,{v_{{1} }}%
^{8}+2\,{v_{{1}}}^{6}{v_{{2}}}^{2}-2\,{v_{{1}}}^{2}{v_{{2}}}^{6}-{v_{{2}}}%
^{8}-10\,{v_{{1}}}^{6}-18\,{v_{{1}}}^{4}{v_{{2}}}^{2}-6\,{v_{{1} }}^{2}%
{v_{{2}}}^{4}+2\,{v_{{2}}}^{6}+33\,{v_{{1}}}^{4}+48\,{v_{{1}}}^{ 2}{v_{{2}}%
}^{2}+15\,{v_{{2}}}^{4}-40\,{v_{{1}}}^{2}-32\,{v_{{2}}}^{2}+ 16,{v_{{1}}}%
^{8}+2\,{v_{{1}}}^{6}{v_{{2}}}^{2}-2\,{v_{{1}}}^{2}{v_{{2} }}^{6}-{v_{{2}}%
}^{8}-10\,{v_{{1}}}^{6}-12\,{v_{{1}}}^{4}{v_{{2}}}^{2}+ 6\,{v_{{1}}}%
^{2}{v_{{2}}}^{4}+8\,{v_{{2}}}^{6}+33\,{v_{{1}}}^{4}+18\,{ v_{{1}}}^{2}%
{v_{{2}}}^{2}-15\,{v_{{2}}}^{4}-40\,{v_{{1}}}^{2}-8\,{v_{{ 2}}}^{2}%
+16,{v_{{1}}}^{8}+2\,{v_{{1}}}^{6}{v_{{2}}}^{2}-2\,{v_{{1}}}^{ 2}{v_{{2}}}%
^{6}-{v_{{2}}}^{8}-7\,{v_{{1}}}^{6}-9\,{v_{{1}}}^{4}{v_{{2} }}^{2}+3\,{v_{{1}%
}}^{2}{v_{{2}}}^{4}+5\,{v_{{2}}}^{6}+15\,{v_{{1}}}^{4 }+12\,{v_{{1}}}%
^{2}{v_{{2}}}^{2}-3\,{v_{{2}}}^{4}-13\,{v_{{1}}}^{2}-5 \,{v_{{2}}}%
^{2}+4,{v_{{1}}}^{8}-4\,{v_{{1}}}^{6}{v_{{2}}}^{2}+6\,{v_{{ 1}}}^{4}{v_{{2}}%
}^{4}-4\,{v_{{1}}}^{2}{v_{{2}}}^{6}+{v_{{2}}}^{8}-10\, {v_{{1}}}%
^{6}+30\,{v_{{1}}}^{4}{v_{{2}}}^{2}-30\,{v_{{1}}}^{2}{v_{{2}} }^{4}%
+10\,{v_{{2}}}^{6}+33\,{v_{{1}}}^{4}-66\,{v_{{1}}}^{2}{v_{{2}}}^{
2}+33\,{v_{{2}}}^{4}-40\,{v_{{1}}}^{2}+40\,{v_{{2}}}^{2}+16,{v_{{1}}}%
^{8}+4\,{v_{{1}}}^{6}{v_{{2}}}^{2}+6\,{v_{{1}}}^{4}{v_{{2}}}^{4}+4\,{v_{{1}}%
}^{2}{v_{{2}}}^{6}+{v_{{2}}}^{8}-10\,{v_{{1}}}^{6}-30\,{v_{{1}}}^{4}{v_{{2}}%
}^{2}-30\,{v_{{1}}}^{2}{v_{{2}}}^{4}-10\,{v_{{2}}}^{6}+33\, {v_{{1}}}%
^{4}+66\,{v_{{1}}}^{2}{v_{{2}}}^{2}+33\,{v_{{2}}}^{4}-40\,{v_{{1}}}%
^{2}-40\,{v_{{2}}}^{2}+16\rangle$

\medskip

$\mathbf{S}=\langle{v_{{1}}}^{4}-2\,{v_{{1}}}^{2}{v_{{2}}}^{2}+{v_{{2}}}%
^{4}-5\, {v_{{1}}}^{2}+5\,{v_{{2}}}^{2}+4,{v_{{1}}}^{4}+2\,{v_{{1}}}%
^{2}{v_{{2} }}^{2}+{v_{{2}}}^{4}-5\,{v_{{1}}}^{2}-5\,{v_{{2}}}^{2}+4,{v_{{1}}%
}^{8} -2\,{v_{{1}}}^{4}{v_{{2}}}^{4}+{v_{{2}}}^{8}-10\,{v_{{1}}}%
^{6}+10\,{v_{{1}}}^{2}{v_{{2}}}^{4}+33\,{v_{{1}}}^{4}-17\,{v_{{2}}}%
^{4}-40\,{v_{{1 }}}^{2}+16,{v_{{1}}}^{8}-4\,{v_{{1}}}^{6}{v_{{2}}}%
^{2}+6\,{v_{{1}}}^{4 }{v_{{2}}}^{4}-4\,{v_{{1}}}^{2}{v_{{2}}}^{6}+{v_{{2}}%
}^{8}-10\,{v_{{1} }}^{6}+30\,{v_{{1}}}^{4}{v_{{2}}}^{2}-30\,{v_{{1}}}%
^{2}{v_{{2}}}^{4}+ 10\,{v_{{2}}}^{6}+33\,{v_{{1}}}^{4}-66\,{v_{{1}}}%
^{2}{v_{{2}}}^{2}+33 \,{v_{{2}}}^{4}-40\,{v_{{1}}}^{2}+40\,{v_{{2}}}%
^{2}+16,{v_{{1}}}^{8}+4 \,{v_{{1}}}^{6}{v_{{2}}}^{2}+6\,{v_{{1}}}^{4}{v_{{2}}%
}^{4}+4\,{v_{{1}} }^{2}{v_{{2}}}^{6}+{v_{{2}}}^{8}-10\,{v_{{1}}}%
^{6}-30\,{v_{{1}}}^{4}{v _{{2}}}^{2}-30\,{v_{{1}}}^{2}{v_{{2}}}^{4}%
-10\,{v_{{2}}}^{6}+33\,{v_{{ 1}}}^{4}+66\,{v_{{1}}}^{2}{v_{{2}}}%
^{2}+33\,{v_{{2}}}^{4}-40\,{v_{{1}} }^{2}-40\,{v_{{2}}}^{2}+16\rangle$

\medskip

\noindent Note that the generators of the two ideals look very different.
However, after computing the reduced Gr\"obner basis~\cite{Buch65, Buch06,
CLO97} of the generators with respect to the total degree order ($v_{1}\succ
v_{2}$), we obtain
\begin{align*}
\mathbf{R}  &  = \langle2\,{v_{1}}^{2}{v_{2}}^{2}-5\,{v_{2}}^{2},\;\; {v_{1}%
}^{4}+{v_{2}}^{4}-5\,{v_{1}}^{2}+4,\;\; 4\,{v_{2}}^{6}-9\,{v_{2}}^{2}\rangle\\
\mathbf{S}  &  = \langle2\,{v_{1}}^{2}{v_{2}}^{2}-5\,{v_{2}}^{2},\;\; {v_{1}%
}^{4}+{v_{2}}^{4}-5\,{v_{1}}^{2}+4,\;\; 4\,{v_{2}}^{6}-9\,{v_{2}}^{2}\rangle
\end{align*}
Thus we have $\mathbf{R}=\mathbf{S}$.
\end{example}


\section{Sketchy Proof of Main Result}

\label{sketch}

The complete proof of the main result (Theorem~\ref{main}) is long and
technical. Hence, before we plunge into the complete proof (given in the next
section), we informally sketch the overall structure/underlying ideas of the
proof by using a small case $m=5,n=4$.

Note that\underline{}%
\[
\mathbf{R}=%
\begin{array}
[c]{cccc}%
\left(  \alpha_{1}+\beta_{1}\right)  & \left(  \alpha_{1}+\beta_{2}\right)  &
\left(  \alpha_{1}+\beta_{3}\right)  & \left(  \alpha_{1}+\beta_{4}\right) \\
\left(  \alpha_{2}+\beta_{1}\right)  & \left(  \alpha_{2}+\beta_{2}\right)  &
\left(  \alpha_{2}+\beta_{3}\right)  & \left(  \alpha_{2}+\beta_{4}\right) \\
\left(  \alpha_{3}+\beta_{1}\right)  & \left(  \alpha_{3}+\beta_{2}\right)  &
\left(  \alpha_{3}+\beta_{3}\right)  & \left(  \alpha_{3}+\beta_{4}\right) \\
\left(  \alpha_{4}+\beta_{1}\right)  & \left(  \alpha_{4}+\beta_{2}\right)  &
\left(  \alpha_{4}+\beta_{3}\right)  & \left(  \alpha_{4}+\beta_{4}\right) \\
\left(  \alpha_{5}+\beta_{1}\right)  & \left(  \alpha_{5}+\beta_{2}\right)  &
\left(  \alpha_{5}+\beta_{3}\right)  & \left(  \alpha_{5}+\beta_{4}\right)
\end{array}
\]%
\[
\mathbf{S}=\mathop{\rm per}%
\begin{bmatrix}
a_{0} & a_{1} & a_{2} & a_{3} & a_{4} & a_{5} &  &  & \\
& a_{0} & a_{1} & a_{2} & a_{3} & a_{4} & a_{5} &  & \\
&  & a_{0} & a_{1} & a_{2} & a_{3} & a_{4} & a_{5} & \\
&  &  & a_{0} & a_{1} & a_{2} & a_{3} & a_{4} & a_{5}\\
b_{0} & b_{1} & b_{2} & b_{3} & b_{4} &  &  &  & \\
& b_{0} & b_{1} & b_{2} & b_{3} & b_{4} &  &  & \\
&  & b_{0} & b_{1} & b_{2} & b_{3} & b_{4} &  & \\
&  &  & b_{0} & b_{1} & b_{2} & b_{3} & b_{4} & \\
&  &  &  & b_{0} & b_{1} & b_{2} & b_{3} & b_{4}%
\end{bmatrix}
\]
Recall that the main theorem (Theorem~\ref{main}) states that $\mathbf{R=S}$,
that is, a term appears in $\mathbf{R}$ if and only if it appears in
$\mathbf{S}$. The overall strategy for the proof is to divide the task into
showing the following four claims:

\begin{enumerate}
\itemsep=0em

\item A term occurs in $\mathbf{R}$ iff it has a representation\ in terms of a
certain boolean matrix, which we call ``res''-representation.

\item A term occurs in $\mathbf{S}$ iff it has a representation in terms of a
certain pair of boolean matrices, which we call ``syl''-representation.

\item If a term has a res-representation then it has a syl-representation.

\item If a term has a syl-representation then it has a res-representation.
\end{enumerate}

The main result (Theorem~\ref{main}) immediately follows from the above four
claims. In the following four subsections, we informally explain/justify the
claims one by one.

\subsection{A term occurs in $\mathbf{R}$ iff it has a res-representation.}

By expanding $\mathbf{R,}$ we obtain%
\[
\mathbf{R}=\cdots+\alpha_{1}^{3}\alpha_{2}^{3}\alpha_{3}^{3}\alpha_{4}%
^{2}\alpha_{5}^{2}\beta_{1}^{2}\beta_{2}^{2}\beta_{3}^{2}\beta_{4}^{1}+\cdots
\]
The particular term above can be obtained by making the underlined choice
while expanding $\mathbf{R.}$%
\[%
\begin{array}
[c]{cccc}%
\left(  \underline{\alpha_{1}}+\beta_{1}\right)  & \left(  \alpha
_{1}+\underline{\beta_{2}}\right)  & \left(  \underline{\alpha_{1}}+\beta
_{3}\right)  & \left(  \underline{\alpha_{1}}+\beta_{4}\right) \\
\left(  \underline{\alpha_{2}}+\beta_{1}\right)  & \left(  \underline
{\alpha_{2}}+\beta_{2}\right)  & \left(  \alpha_{2}+\underline{\beta_{3}%
}\right)  & \left(  \underline{\alpha_{2}}+\beta_{4}\right) \\
\left(  \underline{\alpha_{3}}+\beta_{1}\right)  & \left(  \underline
{\alpha_{3}}+\beta_{2}\right)  & \left(  \alpha_{3}+\underline{\beta_{3}%
}\right)  & \left(  \underline{\alpha_{3}}+\beta_{4}\right) \\
\left(  \alpha_{4}+\underline{\beta_{1}}\right)  & \left(  \alpha
_{4}+\underline{\beta_{2}}\right)  & \left(  \underline{\alpha_{4}}+\beta_{3}
\right)  & \left(  \underline{\alpha_{4}}+\beta_{4}\right) \\
\left(  \alpha_{5}+\underline{\beta_{1}}\right)  & \left(  \underline
{\alpha_{5}}+\beta_{2}\right)  & \left(  \underline{\alpha_{5}}+\beta
_{3}\right)  & \left(  \alpha_{5}+\underline{\beta_{4}}\right)
\end{array}
\]
It is convenient to represent the choice with the following boolean matrix%
\[
\mathcal{R=}\left[
\begin{array}
[c]{cccc}%
1 & 0 & 1 & 1\\
1 & 1 & 0 & 1\\
1 & 1 & 0 & 1\\
0 & 0 & 1 & 1\\
0 & 1 & 1 & 0
\end{array}
\right]
\]
where $1$ in the $i$-th row means that $\alpha_{i}$ is chosen and $0$ in the
$j$-th column means that $\beta_{j}$ is chosen. Then obviously the particular
term above can be written as%
\[
\alpha_{1}^{3}\alpha_{2}^{3}\alpha_{3}^{3}\alpha_{4}^{2}\alpha_{5}^{2}%
\beta_{1}^{2}\beta_{2}^{2}\beta_{3}^{2}\beta_{4}^{1}=\alpha^{rs\left(
\mathcal{R}\right)  }\beta^{cs\left(  \mathcal{\bar{R}}\right)  }%
\]
where $\bar{\mathcal{R}}$ stands for boolean complement of $\mathcal{R}$, $rs$
for row sum and $cs$ for column sum. Considering all other terms in the same
manner, we have%
\[
\mathbf{R}=\sum_{\mathcal{R\in}\left\{  0,1\right\}  ^{5\times4}}%
\alpha^{rs\left(  \mathcal{R}\right)  }\beta^{cs\left(  \mathcal{\bar{R}%
}\right)  }%
\]
Let us call the matrix $\mathcal{R}$ a ``res''-representation of the
corresponding term. In summary, we observe that a term occurs in $\mathbf{R}$
iff it has a res-representation.

\subsection{A term occurs in $\mathbf{S}$ iff it has a syl-representation.}

\bigskip By expanding $\mathbf{S,}$ we obtain%
\[
\mathbf{S}=\cdots+\alpha_{1}^{3}\alpha_{2}^{3}\alpha_{3}^{3}\alpha_{4}%
^{2}\alpha_{5}^{2}\beta_{1}^{2}\beta_{2}^{2}\beta_{3}^{2}\beta_{4}^{1}+\cdots
\]
The particular term above can be obtained by making the following choice
(permutation path) while expanding $\mathbf{S.}$%
\begin{equation}
\label{path1}%
\begin{array}
[c]{|c|c|c|c|c|c|c|c|c|}\hline
& \alpha_{3} &  &  &  &  &  &  & \\\hline
&  &  &  &  & \alpha_{1}\alpha_{2}\alpha_{3}\alpha_{5} &  &  & \\\hline
&  &  &  &  &  &  & \alpha_{1}\alpha_{2}\alpha_{3}\alpha_{4}\alpha_{5} &
\\\hline
&  &  &  &  &  & \alpha_{1}\alpha_{2}\alpha_{4} &  & \\\hline
1 &  &  &  &  &  &  &  & \\\hline
&  & \beta_{2} &  &  &  &  &  & \\\hline
&  &  &  & \beta_{1}\beta_{3} &  &  &  & \\\hline
&  &  & 1 &  &  &  &  & \\\hline
&  &  &  &  &  &  &  & \beta_{1}\beta_{2}\beta_{3}\beta_{4}\\\hline
\end{array}
\tag{*}%
\end{equation}
It is convenient to represent the choice with the following boolean matrices%
\[
\mathcal{S}_{1}=\left[
\begin{array}
[c]{cccc}%
0 & 1 & 1 & 1\\
0 & 1 & 1 & 1\\
1 & 1 & 1 & 0\\
0 & 0 & 1 & 1\\
0 & 1 & 1 & 0
\end{array}
\right]  ,\ \ \mathcal{S}_{2}=\left[
\begin{array}
[c]{cccc}%
0 & 0 & 0 & 0\\
0 & 1 & 0 & 0\\
1 & 0 & 1 & 0\\
0 & 0 & 0 & 0\\
1 & 1 & 1 & 1
\end{array}
\right]
\]
where the $j$-th column of $\mathcal{S}_{1}$ encodes the $\alpha$ term in the
$j$-th row of the top $4 \times9$ submatrix of the above matrix~\eqref{path1}
and the $i$-th row of $\mathcal{S}_{2}$ encodes the $\beta$ term in the $i$-th
row of the bottom $5 \times9$ submatrix of the above matrix~\eqref{path1}.
Then the particular term above can be written as
\[
\alpha_{1}^{3}\alpha_{2}^{3}\alpha_{3}^{3}\alpha_{4}^{2}\alpha_{5}^{2}%
\beta_{1}^{2}\beta_{2}^{2}\beta_{3}^{2}\beta_{4}^{1}=\alpha^{rs\left(
\mathcal{S}_{1}\right)  }\beta^{cs\left(  \mathcal{S}_{2}\right)  }%
\]
Let $c=cs\left(  \mathcal{S}_{1}\right)  $ and $r=rs\left(  \mathcal{S}%
_{2}\right)  .$ Then $c=\left(  1,4,5,3\right)  $ and $r=\left(
0,1,2,0,4\right)  .$ Observe%
\begin{align*}
&  c_{1}+1,c_{2}+2,c_{3}+3,c_{4}+4,r_{1}+1,r_{2}+2,r_{3}+3,r_{4}+4,r_{5}+5\\
&  =2,6,8,7,1,3,5,4,9
\end{align*}
which is exactly the permutation path (the choice of columns) yielding the
particular term. This motivates the following short-hand notations:
\begin{align*}
acs(\mathcal{S}_{1})  &  = (c_{1}+1,c_{2}+2,c_{3}+3,c_{4}+4)\\
ars(\mathcal{S}_{2})  &  = (r_{1}+1,r_{2}+2,r_{3}+3,r_{4}+4,r_{5}+5)
\end{align*}
Using these notations, we can restate the above observation (property) as
\[
\{c^{\prime}_{1},\ldots,c^{\prime}_{4},r^{\prime}_{1},\ldots,r^{\prime}_{5}\}
= \{1,\ldots,4+5\}
\]
where $c^{\prime}=acs(\mathcal{S}_{1})$ and $r^{\prime}=arc(S_{2})$. We will
denote this property by $PC\left(  \mathcal{S}_{1},\mathcal{S}_{2}\right)  $.
Considering all other terms in the same manner, we have%
\[
\mathbf{R}=\sum_{\substack{\mathcal{S}_{1},\mathcal{S}_{2}\mathcal{\in
}\left\{  0,1\right\}  ^{5\times4}\\PC\left(  \mathcal{S}_{1},\mathcal{S}%
_{2}\right)  }}\alpha^{rs\left(  \mathcal{S}_{1}\right)  }\beta^{cs\left(
\mathcal{S}_{2}\right)  }%
\]
Let us call the pair of the matrices ($\mathcal{S}_{1}$, $\mathcal{S}_{2}$) a
``syl''-representation of the corresponding term. In summary, we observe that
a term occurs in $\mathbf{S}$ iff it has a syl-representation.

\subsection{If a term has a res-representation then it has a
syl-representation.}

Consider the term $t=\alpha_{1}^{3}\alpha_{2}^{3}\alpha_{3}^{3}\alpha_{4}%
^{2}\alpha_{5}^{2}\beta_{1}^{2}\beta_{2}^{2}\beta_{3}^{2}\beta_{4}^{1}$ in
$\mathbf{R}$ obtained by the following choice
\[%
\begin{array}
[c]{cccc}%
\left(  \underline{\alpha_{1}}+\beta_{1}\right)  & \left(  \alpha
_{1}+\underline{\beta_{2}}\right)  & \left(  \underline{\alpha_{1}}+\beta
_{3}\right)  & \left(  \underline{\alpha_{1}}+\beta_{4}\right) \\
\left(  \underline{\alpha_{2}}+\beta_{1}\right)  & \left(  \underline
{\alpha_{2}}+\beta_{2}\right)  & \left(  \alpha_{2}+\underline{\beta_{3}%
}\right)  & \left(  \underline{\alpha_{2}}+\beta_{4}\right) \\
\left(  \underline{\alpha_{3}}+\beta_{1}\right)  & \left(  \underline
{\alpha_{3}}+\beta_{2}\right)  & \left(  \alpha_{3}+\underline{\beta_{3}%
}\right)  & \left(  \underline{\alpha_{3}}+\beta_{4}\right) \\
\left(  \alpha_{4}+\underline{\beta_{1}}\right)  & \left(  \alpha
_{4}+\underline{\beta_{2}}\right)  & \left(  \underline{\alpha_{4}}+\beta
_{3}\right)  & \left(  \underline{\alpha_{4}}+\beta_{4}\right) \\
\left(  \alpha_{5}+\underline{\beta_{1}}\right)  & \left(  \underline
{\alpha_{5}}+\beta_{2}\right)  & \left(  \underline{\alpha_{5}}+\beta
_{3}\right)  & \left(  \alpha_{5}+\underline{\beta_{4}}\right)
\end{array}
\]
which is represented by the following res-presentation%
\[
\mathcal{R=}\left[
\begin{array}
[c]{cccc}%
1 & 0 & 1 & 1\\
1 & 1 & 0 & 1\\
1 & 1 & 0 & 1\\
0 & 0 & 1 & 1\\
0 & 1 & 1 & 0
\end{array}
\right]
\]
We construct the following two boolean matrices%
\[
\mathcal{S}_{1}=\left[
\begin{array}
[c]{cccc}%
1 & 0 & 1 & 1\\
1 & 1 & 0 & 1\\
1 & 1 & 0 & 1\\
0 & 0 & 1 & 1\\
0 & 1 & 1 & 0
\end{array}
\right]  ,\ \ \mathcal{S}_{2}=\left[
\begin{array}
[c]{cccc}%
0 & 0 & 0 & 0\\
0 & 0 & 0 & 0\\
0 & 0 & 0 & 0\\
1 & 1 & 1 & 0\\
1 & 1 & 1 & 1
\end{array}
\right]
\]
where $\mathcal{S}_{1}=\mathcal{R}$, $cs\left(  \mathcal{S}_{2}\right)
=cs\left(  \mathcal{\bar{R}}\right)  $ and 1's in $\mathcal{S}_{2}$ are
\textquotedblleft flushed'' to the \ bottom along columns. Observe that
$(\mathcal{S}_{1}, \mathcal{S}_{2})$ represents the following choice yielding
the above term $t$, that is, it is a syl-representation of the term $t$.
\[%
\begin{array}
[c]{|c|c|c|c|c|c|c|c|c|}\hline
&  &  & \alpha_{1}\alpha_{2}\alpha_{3} &  &  &  &  & \\\hline
&  &  &  & \alpha_{2}\alpha_{3}\alpha_{5} &  &  &  & \\\hline
&  &  &  &  & \alpha_{1}\alpha_{4}\alpha_{5} &  &  & \\\hline
&  &  &  &  &  &  & \alpha_{1}\alpha_{2}\alpha_{3}\alpha_{4} & \\\hline
1 &  &  &  &  &  &  &  & \\\hline
& 1 &  &  &  &  &  &  & \\\hline
&  & 1 &  &  &  &  &  & \\\hline
&  &  &  &  &  & \beta_{1}\beta_{2}\beta_{3} &  & \\\hline
&  &  &  &  &  &  &  & \beta_{1}\beta_{2}\beta_{3}\beta_{4}\\\hline
\end{array}
\]
In summary, if a term has a res-representation then it has a syl-representation.

\subsection{If a term has a syl-representation then it has a
res-representation.}

Consider the term $t=\alpha_{1}^{3}\alpha_{2}^{3}\alpha_{3}^{3}\alpha_{4}%
^{2}\alpha_{5}^{2}\beta_{1}^{2}\beta_{2}^{2}\beta_{3}^{2}\beta_{4}^{1}$ in
$\mathbf{S}$ obtained by the following choice%
\[%
\begin{array}
[c]{|c|c|c|c|c|c|c|c|c|}\hline
& \alpha_{3} &  &  &  &  &  &  & \\\hline
&  &  &  &  & \alpha_{1}\alpha_{2}\alpha_{3}\alpha_{5} &  &  & \\\hline
&  &  &  &  &  &  & \alpha_{1}\alpha_{2}\alpha_{3}\alpha_{4}\alpha_{5} &
\\\hline
&  &  &  &  &  & \alpha_{1}\alpha_{2}\alpha_{4} &  & \\\hline
1 &  &  &  &  &  &  &  & \\\hline
&  & \beta_{2} &  &  &  &  &  & \\\hline
&  &  &  & \beta_{1}\beta_{3} &  &  &  & \\\hline
&  &  & 1 &  &  &  &  & \\\hline
&  &  &  &  &  &  &  & \beta_{1}\beta_{2}\beta_{3}\beta_{4}\\\hline
\end{array}
\]
which is represented by the following syl-representation%

\[
\mathcal{S}_{1}=\left[
\begin{array}
[c]{cccc}%
0 & 1 & 1 & 1\\
0 & 1 & 1 & 1\\
1 & 1 & 1 & 0\\
0 & 0 & 1 & 1\\
0 & 1 & 1 & 0
\end{array}
\right]  ,\ \ \mathcal{S}_{2}=\left[
\begin{array}
[c]{cccc}%
0 & 0 & 0 & 0\\
0 & 1 & 0 & 0\\
1 & 0 & 1 & 0\\
0 & 0 & 0 & 0\\
1 & 1 & 1 & 1
\end{array}
\right]
\]
We repeatedly transform the syl-representations of the term $t$ as follows.

\begin{enumerate}
\itemsep=1em

\item We make $acs(\mathcal{S}_{1})$ and $ars(\mathcal{S}_{2})$ increasing.

By swapping

\begin{itemize}
\item $\mathcal{S}_{1,3,3}$ and $\mathcal{S}_{1,3,4}$

\item $\mathcal{S}_{2,3,1}$ and $\mathcal{S}_{2,4,1}$
\end{itemize}

of the above syl-representation, we obtain another syl-representation
\[
\mathcal{S}_{1}=\left[
\begin{array}
[c]{cccc}%
0 & 1 & 1 & 1\\
0 & 1 & 1 & 1\\
1 & 1 & 0 & 1\\
0 & 0 & 1 & 1\\
0 & 1 & 1 & 0
\end{array}
\right]  ,\ \ \mathcal{S}_{2}= \left[
\begin{array}
[c]{cccc}%
0 & 0 & 0 & 0\\
0 & 1 & 0 & 0\\
0 & 0 & 1 & 0\\
1 & 0 & 0 & 0\\
1 & 1 & 1 & 1
\end{array}
\right]
\]
where $\mathcal{S}_{k,i,j}$ denotes the $(i,j)$ entry of the matrix
$\mathcal{S}_{k}$. It represents the following choice.
\[
\begin{array}
[c]{|c|c|c|c|c|c|c|c|c|}\hline
& \alpha_{3} &  &  &  &  &  &  & \\\hline
&  &  &  &  & \alpha_{1}\alpha_{2}\alpha_{3}\alpha_{5} &  &  & \\\hline
&  &  &  &  &  & \alpha_{1}\alpha_{2}\alpha_{4}\alpha_{5} &  & \\\hline
&  &  &  &  &  &  & \alpha_{1}\alpha_{2}\alpha_{3}\alpha_{4} & \\\hline
1 &  &  &  &  &  &  &  & \\\hline
&  & \beta_{2} &  &  &  &  &  & \\\hline
&  &  & \beta_{3} &  &  &  &  & \\\hline
&  &  &  & \beta_{1} &  &  &  & \\\hline
&  &  &  &  &  &  &  & \beta_{1}\beta_{2}\beta_{3}\beta_{4}\\\hline
\end{array}
\]

Note that the column index of the top/bottom part of the matrix is increasing
as the row index is increasing.

\item We make $1$'s in $\mathcal{S}_{2}$ ``flushed to bottom''.

By swapping

\begin{itemize}
\item $\mathcal{S}_{2,2,2}$ and $\mathcal{S}_{2,3,2}$

\item $\mathcal{S}_{2,3,2}$ and $\mathcal{S}_{2,4,2}$

\item $\mathcal{S}_{1,1,1}$ and $\mathcal{S}_{1,1,2}$

\item $\mathcal{S}_{2,3,3}$ and $\mathcal{S}_{2,4,3}$

\item $\mathcal{S}_{1,2,1}$ and $\mathcal{S}_{1,2,3}$
\end{itemize}

we obtain still another syl-representation
\[
\mathcal{S}_{1}=\left[
\begin{array}
[c]{cccc}%
1 & 0 & 1 & 1\\
1 & 1 & 0 & 1\\
1 & 1 & 0 & 1\\
0 & 0 & 1 & 1\\
0 & 1 & 1 & 0
\end{array}
\right]  ,\ \ \mathcal{S}_{2}=\ \ \left[
\begin{array}
[c]{cccc}%
0 & 0 & 0 & 0\\
0 & 0 & 0 & 0\\
0 & 0 & 0 & 0\\
1 & 1 & 1 & 0\\
1 & 1 & 1 & 1
\end{array}
\right]
\]
It represents the following choice.
\[
\begin{array}
[c]{|c|c|c|c|c|c|c|c|c|}\hline
&  &  & \alpha_{1}\alpha_{2}\alpha_{3} &  &  &  &  & \\\hline
&  &  &  & \alpha_{2}\alpha_{3}\alpha_{5} &  &  &  & \\\hline
&  &  &  &  & \alpha_{1}\alpha_{4}\alpha_{5} &  &  & \\\hline
&  &  &  &  &  &  & \alpha_{1}\alpha_{2}\alpha_{3}\alpha_{4} & \\\hline
1 &  &  &  &  &  &  &  & \\\hline
& 1 &  &  &  &  &  &  & \\\hline
&  & 1 &  &  &  &  &  & \\\hline
&  &  &  &  &  & \beta_{1}\beta_{2}\beta_{3} &  & \\\hline
&  &  &  &  &  &  &  & \beta_{1}\beta_{2}\beta_{3}\beta_{4}\\\hline
\end{array}
\]
Note that the entries in the bottom part of the matrix are products of the
form $\beta_{1}\cdots\beta_{k}$ (product of consecutive $\beta$'s starting
from $\beta_{1}$).
\end{enumerate}

\noindent Let $\mathcal{R=S}_{1}.$ Observe that $\mathcal{R}$ represents the
following choice yielding the given term $t$, that is, it is a
res-representation of the term $t$.%
\[%
\begin{array}
[c]{cccc}%
\left(  \underline{\alpha_{1}}+\beta_{1}\right)  & \left(  \alpha
_{1}+\underline{\beta_{2}}\right)  & \left(  \underline{\alpha_{1}}+\beta
_{3}\right)  & \left(  \underline{\alpha_{1}}+\beta_{4}\right) \\
\left(  \underline{\alpha_{2}}+\beta_{1}\right)  & \left(  \underline
{\alpha_{2}}+\beta_{2}\right)  & \left(  \alpha_{2}+\underline{\beta_{3}%
}\right)  & \left(  \underline{\alpha_{2}}+\beta_{4}\right) \\
\left(  \underline{\alpha_{3}}+\beta_{1}\right)  & \left(  \underline
{\alpha_{3}}+\beta_{2}\right)  & \left(  \alpha_{3}+\underline{\beta_{3}%
}\right)  & \left(  \underline{\alpha_{3}}+\beta_{4}\right) \\
\left(  \alpha_{4}+\underline{\beta_{1}}\right)  & \left(  \alpha
_{4}+\underline{\beta_{2}}\right)  & \left(  \underline{\alpha_{4}}+\beta_{3}
\right)  & \left(  \underline{\alpha_{4}}+\beta_{4}\right) \\
\left(  \alpha_{5}+\underline{\beta_{1}}\right)  & \left(  \underline
{\alpha_{5}}+\beta_{2}\right)  & \left(  \underline{\alpha_{5}}+\beta
_{3}\right)  & \left(  \alpha_{5}+\underline{\beta_{4}}\right)
\end{array}
\]
In summary, if a term has a syl-representation then it has a res-representation.


\section{Detailed Proof of Main Result}

\label{proof} In this section, we provide a detailed proof for the main
result. We strongly recommend that the reader first go over the previous
section (Section~\ref{sketch}) where we provided an informal/intuitive sketch
of the overall structure/underlying ideas of the proof. It will greatly aid
the reader in following the detailed and technical proof given below.

\medskip

Recall that the main theorem (Theorem~\ref{main}) claims that $\mathbf{R=S}$,
that is, a term appears in $\mathbf{R}$ if and only if it appears in
$\mathbf{S}$. It follows immediately from the following four lemmas.

\begin{description}
[leftmargin=6em,style=nextline]

\item[\textsf{Lemma \ref{Delta_res}:}] A term occurs in $\mathbf{R}$ iff it
has a res-representation (a certain boolean matrix).

\item[\textsf{Lemma \ref{Pi_syl}:}] A term occurs in $\mathbf{S}$ iff it has a
syl-representation (a certain pair of boolean matrices).

\item[\textsf{Lemma \ref{SylFromRes}:}] If a term has a res-representation
then it has a syl-representation.

\item[\textsf{Lemma \ref{ResFromSyl}:}] If a term has a syl-representation
then it has a res-representation.
\end{description}

\noindent We will devote one subsection for each lemma. Each subsection ends
with the proof of each of the above lemmas. Lemmas~\ref{SylFromRes}
and~\ref{ResFromSyl} are proved constructively, by providing algorithms
(Algorithms~\ref{alg:SylFromRes} and~\ref{alg:ResFromSyl}) that produce one
representation from the other. These algorithms are based on a key lemma
(Lemma~\ref{lem:Flushed}) that establishes a crucial relationship between the
two representations (syl and res). Before stating and proving the above
lemmas, we introduce notations that will be used throughout this section.

\begin{notation}
\label{notations}Let $M\in\left\{  0,1\right\}  ^{m\times n}$.

\begin{enumerate}
\item The \emph{complement} of $M,$ written as $\bar{M},$ is the $m\times n$
matrix defined by%
\[
\bar{M}_{ij}=1-M_{ij}%
\]

\item The \emph{row sum} of $M,$ written as $rs\left(  M\right)  $, is the
$m$-dimensional vector defined by%
\[
rs\left(  M\right)  =\left(  \sum_{j=1}^{n}M_{ij}\right)  _{i=1,\ldots,m}%
\]

\item The \emph{column sum} of $M,$ written as $cs\left(  M\right)  $, is the
$n$-dimensional vector defined by%
\[
cs\left(  M\right)  =\left(  \sum_{i=1}^{m}M_{ij}\right)  _{j=1,\ldots,n}%
\]

\item The \emph{adjusted row sum} of $M,$ written as $ars\left(  M\right)  $,
is the $m$-dimensional vector defined by%
\[
ars\left(  M\right)  =\left(  i+\sum_{j=1}^{n}M_{ij}\right)  _{i=1,\ldots,m}%
\]

\item The \emph{adjusted column sum} of $M,$ written as $acs\left(  M\right)
$, is the $n$-dimensional vector defined by%
\[
acs\left(  M\right)  =\left(  j+\sum_{i=1}^{m}M_{ij}\right)  _{j=1,\ldots,n}%
\]

\end{enumerate}
\end{notation}

\begin{example}
Let%
\[
M=\left[
\begin{array}
[c]{ccc}%
1 & 1 & 1\\
1 & 0 & 1\\
0 & 1 & 1\\
0 & 1 & 0
\end{array}
\right]
\]
Then%
\begin{align*}
rs\left(  M\right)   &  =\left(  3,2,2,1\right) \\
cs\left(  M\right)   &  =\left(  2,3,3\right) \\
ars\left(  M\right)   &  =\left(  4,4,5,5\right) \\
acs\left(  M\right)   &  =\left(  3,5,6\right)
\end{align*}

\end{example}

\subsection{A term occurs in $\mathbf{R}$ iff it has a res-representation.}

\begin{definition}
[Res-representation]\label{res-rep}Let $t=\alpha^{\mu}\beta^{\nu}\!$ be a
term, where $\mu\in\mathbb{Z}^{m}_{{}\ge0} \, \mbox{and} \hskip 4pt \nu
\in\mathbb{Z}_{{}\ge0}^{n}$. Let $\mathcal{R}\in\left\{  0,1\right\}
^{m\times n}\!$. We say that $\mathcal{R}$ is a res-representation of $t$ if

\begin{itemize}
\item $rs\left(  \mathcal{R}\right)  =\mu.$

\item $cs\left(  \bar{\mathcal{R}}\right)  =\nu.$ Equivalently $cs\left(
\mathcal{R}\right)  =\bar{\nu}$ where $\bar{\nu}_{j}=m-\nu_{j}.$
\end{itemize}
\end{definition}

\begin{example}
Let $m=3$ and $n=2.$ Let $t=\alpha_{1}^{2}\alpha_{2}^{1}\alpha_{3}^{1}%
\beta_{1}^{1}\beta_{2}^{1}=\alpha^{\left(  2,1,1\right)  }\beta^{\left(
1,1\right)  }.$ Let
\[
\mathcal{R}=\left[
\begin{array}
[c]{cc}%
1 & 1\\
1 & 0\\
0 & 1
\end{array}
\right]
\]
We would like to know whether $\mathcal{R}$ is a res-representation of $t.$
Note
\[
\mathcal{R}=\left[
\begin{array}
[c]{cc||cc}%
1 & 1 &  & 2\\
1 & 0 &  & 1\\
0 & 1 &  & 1\\\hline\hline
&  &  & \mu_{i}\\
2 & 2 & \bar{\nu}_{j} & \\
1 & 1 & \nu_{j} &
\end{array}
\right]
\]
Hence $\mathcal{R}$ is a res-representation of $t.$ Likewise, the following
matrix is also a res-representation of $t.$%
\[
\mathcal{R}=\left[
\begin{array}
[c]{cc}%
1 & 1\\
0 & 1\\
1 & 0
\end{array}
\right]
\]

\end{example}

\begin{lemma}
\label{Delta_res}A term occurs in $\mathbf{R}$ iff it has a res-representation.
\end{lemma}

\begin{proof}
Recall
\[
\mathbf{R}=\prod_{i=1}^{m}\prod_{j=1}^{n}\left(  \alpha_{i}+\beta_{j}\right)
\]
Note%
\begin{align*}
\mathbf{R}  &  =\prod_{i=1}^{m}\prod_{j=1}^{n}\sum_{e\in\left\{  0,1\right\}
}\alpha_{i}^{e}\beta_{j}^{\bar{e}}\ \ \ \ \ \ \ \ \text{where }\bar{e}=1-e\\
&  =\sum_{\mathcal{R}\in\left\{  0,1\right\}  ^{m\times n}}\prod_{i=1}%
^{m}\prod_{j=1}^{n}\alpha_{i}^{\mathcal{R}_{ij}}\beta_{j}^{\bar{\mathcal{R}%
}_{ij}}\\
&  =\sum_{\mathcal{R}\in\left\{  0,1\right\}  ^{m\times n}}\prod_{i=1}%
^{m}\prod_{j=1}^{n}\alpha_{i}^{\mathcal{R}_{ij}}\ \prod_{i=1}^{m}\prod
_{j=1}^{n}\beta_{j}^{\bar{\mathcal{R}}_{ij}}\\
&  =\sum_{\mathcal{R}\in\left\{  0,1\right\}  ^{m\times n}}\prod_{i=1}%
^{m}\alpha_{i}^{\sum_{j=1}^{n}\mathcal{R}_{ij}}\prod_{j=1}^{n}\beta_{j}%
^{\sum_{i=1}^{m}\bar{\mathcal{R}}_{ij}}\\
&  =\sum_{\mathcal{R}\in\left\{  0,1\right\}  ^{m\times n}}\alpha^{\mu}%
\beta^{\nu}\ \text{where }\mu=rs\left(  \mathcal{R}\right)  \ \text{and }%
\nu=cs\left(  \bar{\mathcal{R}}\right)
\end{align*}
The claim follows immediately.
\end{proof}

\subsection{A term occurs in $\mathbf{S}$ iff it has a syl-representation.}

\begin{definition}
[Properly coupled]Let $\mathcal{S}_{1}, \mathcal{S}_{2}\in\left\{
0,1\right\}  ^{m\times n}\!$. We say that $\mathcal{S}_{1}$ and $\mathcal{S}%
_{2}$ are properly coupled and write as $PC\left(  \mathcal{S}_{1}%
,\mathcal{S}_{2}\right)  $ iff
\[
\left\{  c_{1}^{\prime},\ldots,c_{n}^{\prime},r_{1}^{\prime},\ldots
,r_{m}^{\prime}\right\}  =\left\{  1,\ldots,m+n\right\}
\]
where $c^{\prime}=acs\left(  \mathcal{S}_{1}\right)  $ and $r^{\prime
}=ars\left(  \mathcal{S}_{2}\right)  .$
\end{definition}

\begin{example}
Let
\[%
\begin{array}
[t]{cc}%
\mathcal{S}_{1}=\left[
\begin{array}
[c]{cc}%
1 & 1\\
0 & 1\\
0 & 1
\end{array}
\right]  & \mathcal{S}_{2}=\left[
\begin{array}
[c]{cc}%
0 & 0\\
1 & 0\\
0 & 1
\end{array}
\right]
\end{array}
\]
Since%
\begin{align*}
c^{\prime}  &  =acs\left(  \mathcal{S}_{1}\right)  =\left(  2,5\right) \\
r^{\prime}  &  =ars\left(  \mathcal{S}_{2}\right)  =\left(  1,3,4\right)
\end{align*}
we have%
\[
\left\{  2,5,1,3,4\right\}  =\left\{  1,2,3,4,5\right\}
\]
Hence we have $PC(\mathcal{S}_{1},\mathcal{S}_{2})$.
\end{example}

\begin{definition}
[Syl-representation]\label{syl-rep}Let $t=\alpha^{\mu}\beta^{\nu}$ be a term.
Let $\mathcal{S}_{1},\mathcal{S}_{2}\in\left\{  0,1\right\}  ^{m\times n}$. We
say that $\left(  \mathcal{S}_{1},\mathcal{S}_{2}\right)  $ is a
syl-representation of $t$ if

\begin{itemize}
\item $rs\left(  \mathcal{S}_{1}\right)  =\mu$

\item $cs\left(  \mathcal{S}_{2}\right)  =\nu$

\item $PC\left(  \mathcal{S}_{1},\mathcal{S}_{2}\right)  \ $
\end{itemize}
\end{definition}

\begin{example}
Let $m=3$ and $n=2.$ Let $t=\alpha_{1}^{2}\alpha_{2}^{1}\alpha_{3}^{1}%
\beta_{1}^{1}\beta_{2}^{1}=\alpha^{\left(  2,1,1\right)  }\beta^{\left(
1,1\right)  }.$ Let%
\[%
\begin{array}
[t]{cc}%
\mathcal{S}_{1}=\left[
\begin{array}
[c]{cc}%
1 & 1\\
0 & 1\\
0 & 1
\end{array}
\right]  & \mathcal{S}_{2}=\left[
\begin{array}
[c]{cc}%
0 & 0\\
1 & 0\\
0 & 1
\end{array}
\right]
\end{array}
\]
We would like to know whether $\left(  \mathcal{S}_{1},\mathcal{S}_{2}\right)
$ is a syl-representation of $t.$ Note%
\[%
\begin{array}
[t]{cc}%
\mathcal{S}_{1}=\left[
\begin{array}
[c]{cc||cccc}%
1 & 1 &  & 2 &  & \\
0 & 1 &  & 1 &  & \\
0 & 1 &  & 1 &  & \\\hline\hline
&  &  & \mu_{i} &  & \\
1 & 3 & c_{j} &  &  & \\
1 & 2 & j &  &  & \\
2 & 5 & c_{j}+j &  &  &
\end{array}
\right]  & \mathcal{S}_{2}=\left[
\begin{array}
[c]{cc||cccc}%
0 & 0 &  & 0 & 1 & 1\\
1 & 0 &  & 1 & 2 & 3\\
0 & 1 &  & 1 & 3 & 4\\\hline\hline
&  &  & r_{i} & i & r_{i}+i\\
1 & 1 & \nu_{j} &  &  & \\
&  &  &  &  & \\
&  &  &  &  &
\end{array}
\right]
\end{array}
\]
Hence $\left(  \mathcal{S}_{1},\mathcal{S}_{2}\right)  $ is a
syl-representation of $t.$ Likewise, the following pairs of matrices are also
syl-representations of $t.$%
\begin{align*}
&
\begin{array}
[t]{cc}%
\mathcal{S}_{1}=\left[
\begin{array}
[c]{cc}%
1 & 1\\
0 & 1\\
0 & 1
\end{array}
\right]  & \mathcal{S}_{2}=\left[
\begin{array}
[c]{cc}%
0 & 0\\
0 & 1\\
1 & 0
\end{array}
\right]
\end{array}
\\
&
\begin{array}
[t]{cc}%
\mathcal{S}_{1}=\left[
\begin{array}
[c]{cc}%
1 & 1\\
0 & 1\\
0 & 1
\end{array}
\right]  & \mathcal{S}_{2}=\left[
\begin{array}
[c]{cc}%
0 & 0\\
1 & 1\\
0 & 0
\end{array}
\right]
\end{array}
\\
&
\begin{array}
[t]{cc}%
\mathcal{S}_{1}=\left[
\begin{array}
[c]{cc}%
1 & 1\\
1 & 0\\
0 & 1
\end{array}
\right]  & \mathcal{S}_{2}=\left[
\begin{array}
[c]{cc}%
0 & 0\\
0 & 0\\
1 & 1
\end{array}
\right]
\end{array}
\\
&
\begin{array}
[t]{cc}%
\mathcal{S}_{1}=\left[
\begin{array}
[c]{cc}%
1 & 1\\
0 & 1\\
1 & 0
\end{array}
\right]  & \mathcal{S}_{2}=\left[
\begin{array}
[c]{cc}%
0 & 0\\
0 & 0\\
1 & 1
\end{array}
\right]
\end{array}
\\
&
\begin{array}
[t]{cc}%
\mathcal{S}_{1}=\left[
\begin{array}
[c]{cc}%
1 & 1\\
1 & 0\\
1 & 0
\end{array}
\right]  & \mathcal{S}_{2}=\left[
\begin{array}
[c]{cc}%
0 & 0\\
0 & 0\\
1 & 1
\end{array}
\right]
\end{array}
\end{align*}

\end{example}

\begin{lemma}
\label{Pi_syl}A term occurs in $\mathbf{S}$ iff it has a syl-representation.
\end{lemma}

\begin{proof}
Let $M$ be the Sylvester matrix of $f$ and $g$ (see Definition~\ref{sylmat}).
Then
\begin{align*}
\mathbf{S}  &  =\mathop{\rm per}\left(  M\right) \\
&  =\sum_{\left\{  \sigma_{1},\ldots,\sigma_{n+m}\right\}  =\left\{
1,\ldots,n+m\right\}  }\prod_{i=1}^{n+m}M_{i\sigma_{i}}\\
&  =\sum_{\left\{  \sigma_{1},\ldots,\sigma_{n},\sigma_{n+1},\ldots
,\sigma_{n+m}\right\}  =\left\{  1,\ldots,n+m\right\}  }\left(  \prod
_{j=1}^{n}M_{j\sigma_{j}}\right)  \left(  \prod_{i=1}^{m}M_{\left(
n+i\right)  \sigma_{n+i}}\right) \\
&  =\sum_{\left\{  \sigma_{1},\ldots,\sigma_{n},\sigma_{n+1},\ldots
,\sigma_{n+m}\right\}  =\left\{  1,\ldots,n+m\right\}  }\left(  \prod
_{j=1}^{n}a_{\sigma_{j}-j}\right)  \left(  \prod_{i=1}^{m}b_{\sigma_{n+i}%
-i}\right) \\
&  =\sum_{\substack{\left\{  c_{1}+1,\ldots,c_{n}+n,r_{1}+1,\ldots
,r_{m}+m\right\}  =\left\{  1,\ldots,n+m\right\}  }}\left(  \prod_{j=1}%
^{n}a_{c_{j}}\right)  \left(  \prod_{i=1}^{m}b_{r_{i}}\right)  ,\ \ \text{by
reindexing with }c_{j}=\sigma_{j}-j\ \ \,\text{and }r_{j}=\sigma_{n+i}-i\\
&  =\sum_{\substack{\left\{  c_{1}+1,\ldots,c_{n}+n,r_{1}+1,\ldots
,r_{m}+m\right\}  =\left\{  1,\ldots,n+m\right\}  }}\left(  \prod_{j=1}%
^{n}\sum_{\substack{S_{1}\in\left\{  0,1\right\}  ^{n}\\\sum_{k=1}^{n}%
S_{1,k}=c_{j}}}\prod_{i=1}^{m}\alpha_{i}^{S_{1,i}}\right)  \left(  \prod
_{i=1}^{m}\sum_{\substack{S_{2}\in\left\{  0,1\right\}  ^{n}\\\sum_{k=1}%
^{m}S_{2,k}=r_{j}}}\prod_{j=1}^{n}\beta_{j}^{S_{2,j}}\right) \\
&  =\sum_{\substack{\left\{  c_{1}+1,\ldots,c_{n}+n,r_{1}+1,\ldots
,r_{m}+m\right\}  =\left\{  1,\ldots,n+m\right\}  }}\left(  \sum
_{\substack{\mathcal{S}_{1}\in\left\{  0,1\right\}  ^{m\times n}\\cs\left(
\mathcal{S}_{1}\right)  =c}}\prod_{i=1}^{m}\prod_{j=1}^{n}\alpha
_{i}^{\mathcal{S}_{1,i,j}}\right)  \left(  \sum_{\substack{\mathcal{S}_{2}%
\in\left\{  0,1\right\}  ^{m\times n}\\rs\left(  \mathcal{S}_{2}\right)
=r}}\prod_{i=1}^{m}\prod_{j=1}^{n}\beta_{j}^{\mathcal{S}_{2,i,j}}\right) \\
&  =\sum_{\substack{\mathcal{S}_{1},\mathcal{S}_{2}\in\left\{  0,1\right\}
^{m\times n}\\cs\left(  \mathcal{S}_{1}\right)  =c\\rs\left(  \mathcal{S}%
_{2}\right)  =r\\\left\{  c_{1}+1,\ldots,c_{n}+n,r_{1}+1,\ldots,r_{m}%
+m\right\}  =\left\{  1,\ldots,n+m\right\}  }} \prod_{i=1}^{m}\prod_{j=1}%
^{n}\alpha_{i}^{\mathcal{S}_{1,i,j}}\prod_{i=1}^{m}\prod_{j=1}^{n}\beta
_{j}^{\mathcal{S}_{2,i,j}}\\
&  =\sum_{\substack{\mathcal{S}_{1},\mathcal{S}_{2}\in\left\{  0,1\right\}
^{m\times n}\\PC\left(  \mathcal{S}_{1},\mathcal{S}_{2}\right)  }}\prod
_{i=1}^{m}\prod_{j=1}^{n}\alpha_{i}^{\mathcal{S}_{1,i,j}}\prod_{j=1}^{n}%
\prod_{i=1}^{m}\beta_{j}^{\mathcal{S}_{2,i,j}}\\
&  =\sum_{\substack{\mathcal{S}_{1},\mathcal{S}_{2}\in\left\{  0,1\right\}
^{m\times n}\\PC\left(  \mathcal{S}_{1},\mathcal{S}_{2}\right)  }}\prod
_{i=1}^{m}\alpha_{i}^{\sum_{j=1}^{n}\mathcal{S}_{1,i,j}}\prod_{j=1}^{n}%
\beta_{j}^{\sum_{i=1}^{m}\mathcal{S}_{2,i,j}}\\
&  =\sum_{\substack{\mathcal{S}_{1},\mathcal{S}_{2}\in\left\{  0,1\right\}
^{m\times n}\\PC\left(  \mathcal{S}_{1},\mathcal{S}_{2}\right)  }}\alpha^{\mu
}\beta^{\nu}\ \ \ \text{where }\mu=rs\left(  \mathcal{S}_{1}\right)
\ \text{and }\nu=cs\left(  \mathcal{S}_{2}\right)
\end{align*}
The claim follows immediately.
\end{proof}

\subsection{If a term has a res-representation then it has a
syl-representation.}

The proof is constructive, that is, it provides an algorithm that takes a
res-representation and produces a syl-representation
(Algorithm~\ref{alg:SylFromRes}). The algorithm is immediate from a key lemma
(Lemma~\ref{lem:Flushed}) that establishes a crucial relationship between the
two representations (syl and res). Thus most of this subsection will be
devoted in stating and proving the key lemma.

Note that $\mathbf{R}$ is a symmetric expression in $\alpha_{1},\ldots
,\alpha_{m}$ and in $\beta_{1},\ldots,\beta_{n}.$ Thus for any term
in~$\mathbf{R}$, every term obtained by permuting $\alpha_{1},\ldots
,\alpha_{m}$ and permuting $\beta_{1},\ldots,\beta_{n}$ is also in
$\mathbf{R.\ }$The same holds for~$\mathbf{S}$ too. Hence, without loss of
generality, we may restrict the proof to the terms $\alpha_{1}^{\mu_{1}}%
\cdots\alpha_{m}^{\mu_{m}}\beta_{1}^{\nu_{1}}\cdots\beta_{n}^{\nu_{n}}$ where
$\mu_{1}\geq\mu_{2}\geq\cdots\geq\mu_{m}$ and $\nu_{1}\geq\nu_{2}\geq
\cdots\geq\nu_{n}.$ Therefore, from now on, we will assume that $\mu$ and
$\nu$ are in non-increasing order.

\begin{definition}
[Bottom-left flushed]\label{def:blf} A matrix is called \emph{bottom--left
flushed} if all the non-zero entries are flushed to the bottom--left. \ Let
$c\in\left\{  0,\ldots,m\right\}  ^{n}$. Then the \emph{bottom--left flushed}
matrix of $c$, written as $F_{c}\in\left\{  0,1\right\}  ^{m\times n},$ is the
\emph{bottom--left flushed matrix }such that $cs(F_{c})=c$.
\end{definition}

\begin{example}
Let
\[
M=\left[
\begin{array}
[c]{ccccc}%
0 & 0 & 0 & 0 & 0\\
0 & 0 & 0 & 0 & 0\\
1 & 0 & 0 & 0 & 0\\
1 & 0 & 0 & 0 & 0\\
1 & 1 & 1 & 0 & 0\\
1 & 1 & 1 & 1 & 0
\end{array}
\right]
\]
Then $M$ is flushed. Note also that $M=F_{\left(  4,2,2,1,0\right)  }$.
\end{example}

\begin{lemma}
\label{lem:Flushed_PC} Let $M\in\left\{  0,1\right\}  ^{m\times n}$ be
bottom-left flushed. Then we have $PC\left(  \bar{M},M\right)  $.
\end{lemma}

\begin{proof}
We will prove by mathematical induction on $m+n$. If $m+n=0$ (i.e. $m=n=0$),
then the implication holds vacuously. Now, let us assume that the implication
holds for all bottom-left flushed $m\times n$ matrix such that $m+n$ $<k$.
Consider an arbitrary bottom-left flushed $m\times n$ matrix $M$ such that
$m+n=k.$ Let $c=cs(\bar{M})$ and $r=rs(M)$. We consider two cases as in the
following two figures,

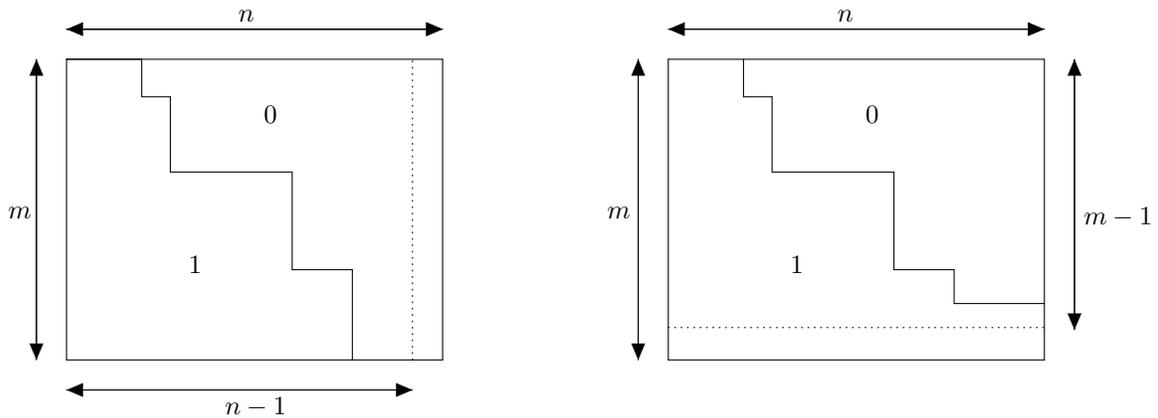
\usetikzlibrary{arrows} \pagestyle{empty}

\begin{tikzpicture}[line cap=round,line join=round,>=triangle 45,x=1.0cm,y=1.0cm]
\clip(-1.14,-1.162) rectangle (14.59,5.02);
\draw (0.,4.)-- (5.,4.);
\draw (5.,4.)-- (5.,0.);
\draw (0.,4.)-- (0.,0.);
\draw (0.,0.)-- (5.,0.);
\draw (8.,4.)-- (13.,4.);
\draw (8.,4.)-- (8.,0.);
\draw (8.,0.)-- (13.,0.);
\draw (13.,0.)-- (13.,4.);
\draw (0.,4.)-- (1.,4.);
\draw (1.,4.)-- (1.,3.5);
\draw (1.,3.5)-- (1.38,3.5);
\draw (1.38,3.5)-- (1.38,2.5);
\draw (1.38,2.5)-- (3.,2.5);
\draw (3.,2.5)-- (3.,1.2);
\draw (3.,1.2)-- (3.8,1.2);
\draw (3.8,1.2)-- (3.8,0.);
\draw (9.,4.)-- (9.,3.5);
\draw (9.,3.5)-- (9.38,3.5);
\draw (9.38,3.5)-- (9.38,2.5);
\draw (9.38,2.5)-- (11.,2.5);
\draw (11.,2.5)-- (11.,1.2);
\draw (11.,1.2)-- (11.80,1.2);
\draw (11.80,1.2)-- (11.80,0.75);
\draw (11.80,0.75)-- (13.,0.75);
\draw [dotted] (4.6,0.)-- (4.6,4.);
\draw [dotted] (8.,0.43)-- (13.,0.43);
\draw [->] (-0.4,4.) -- (-0.4,0.);
\draw [->] (-0.4,0.) -- (-0.4,4.);
\draw [->] (0.,-0.4) -- (4.6,-0.4);
\draw [->] (4.6,-0.4) -- (0.,-0.4);
\draw [->] (0.,4.4) -- (5.,4.4);
\draw [->] (5.,4.4) -- (0.,4.4);
\draw [->] (0.,4.4) -- (5.,4.4);
\draw [->] (8.,4.4) -- (13.,4.4);
\draw [->] (13.,4.4) -- (8.,4.4);
\draw [->] (7.6,4.) -- (7.6,0.);
\draw [->] (7.6,0.) -- (7.6,4.);
\draw [->] (13.4,4.) -- (13.4,0.4);
\draw [->] (13.4,0.4) -- (13.4,4.);
\draw (13.400,2.158) node[anchor=north west] {$m-1$};
\draw (-0.900,2.158) node[anchor=north west] {$m$};
\draw (7.065,2.158) node[anchor=north west] {$m$};
\draw (2.16,4.780) node[anchor=north west] {$n$};
\draw (10.125,4.780) node[anchor=north west] {$n$};
\draw (1.984,-0.379) node[anchor=north west] {$n-1$};
\draw (2.500,3.500) node[anchor=north west] {0};
\draw (1.500,1.500) node[anchor=north west] {1};
\draw (10.500,3.500) node[anchor=north west] {0};
\draw ( 9.500,1.500) node[anchor=north west] {1};
\end{tikzpicture}

\noindent where we have all 0's above the ``stairs'' (jagged sold lines) and
all 1's below the stairs.

\begin{description}
[leftmargin=8em,style=nextline,itemsep=0.0em]

\item[Case $M_{mn}=0$.] Since $M$ is bottom-left flushed, the last column of
$M$ is all zero like the above left figure. Let $M^{\ast}$ be the
$m\times(n-1)$ matrix obtained from $M$ by deleting the last column of $M$.
Note that $M^{\ast}$ is also bottom-left bottom-left flushed. Let $c^{\ast
}=cs(\overline{M^{\ast}})$ and $r^{\ast}=rs(M^{\ast})$. Then
\[
r=r^{\ast}~\text{ and }c=\left(  c_{1}^{\ast},\ldots,c_{n-1}^{\ast},m\right)
.
\]
(In the above, $c$ and $c^{\ast}$ are the column sum vectors of $\bar{M}$ and
$\overline{M^{\ast}}$, respectively, and hence they count the number of 0 on
the columns of $M$ and $M^{\ast}$). Thus%
\begin{align*}
&  \left\{  c_{1}+1,\ldots,c_{n}+n,r_{1}+1,\ldots,r_{m}+m\right\} \\
&  =\left\{  c_{1}^{\ast}+1,\ldots,c_{n-1}^{\ast}+n-1,m+n,r_{1}^{\ast
}+1,\ldots,r_{m}^{\ast}+m\right\} \\
&  =\left\{  c_{1}^{\ast}+1,\ldots,c_{n-1}^{\ast}+n-1,r_{1}^{\ast}%
+1,\ldots,r_{m}^{\ast}+m\right\}  \cup\left\{  m+n\right\} \\
&  =\left\{  1,\ldots,m+n-1\right\}  \cup\left\{  m+n\right\}
\ \ \ \ \text{from the induction hypothesis}\\
&  =\left\{  1,\ldots,m+n\right\}  .
\end{align*}

\item[Case $M_{mn}=1$.] Since $M$ is bottom-left flushed, the last row of $M$
is all one like the above right figure. Let $M^{\ast}$ be the $(m-1)\times
n\ $matrix obtained from $M$ by deleting the last row of $M$. Note that
$M^{\ast}$ is also bottom-left flushed. Let $c^{\ast}=cs(\overline{M^{\ast}})$
and $r^{\ast}=rs(M^{\ast})$. Note that
\[
c=c^{\ast}\text{ and }r=\left(  r_{1}^{\ast},\ldots,r_{m-1}^{\ast},n\right)
.
\]
Thus%
\begin{align*}
&  \left\{  c_{1}+1,\ldots,c_{n}+n,r_{1}+1,\ldots,r_{m}+m\right\} \\
&  =\left\{  c_{1}^{\ast}+1,\ldots,c_{n}^{\ast}+n,r_{1}^{\ast}+1,\ldots
,r_{m-1}^{\ast}+m-1,n+m\right\} \\
&  =\left\{  c_{1}^{\ast}+1,\ldots,c_{n}^{\ast}+n,r_{1}^{\ast}+1,\ldots
,r_{m-1}^{\ast}+m-1\right\}  \cup\left\{  n+m\right\} \\
&  =\left\{  1,\ldots,m-1+n\right\}  \cup\left\{  n+m\right\}  \ \ \text{from
the induction hypothesis}\\
&  =\left\{  1,\ldots,m+n\right\}  .
\end{align*}

\end{description}

\noindent Therefore, in both cases, $PC(\bar{M},M)$ holds.
\end{proof}

\begin{definition}
[Sorted/Flushed]\label{def:sorted-flushed} Let $A,B\in\left\{  0,1\right\}
^{m\times n}$. We say that $(A,B)$ is \emph{sorted} iff $acs(A)$ and $ars(B)$
are sorted in increasing order. We say that $(A,B)$ is \emph{flushed} iff $B$
is bottom-left flushed. We say that $(A,B)$ is \emph{sorted-flushed} iff it is
both sorted and flushed.
\end{definition}

\begin{example}
Let
\[
(A,B) =\left[
\begin{array}
[c]{ccc||cc}%
1 & 0 & 1 &  & \\
0 & 1 & 1 &  & \\
0 & 1 & 1 &  & \\
1 & 0 & 0 &  & \\\hline\hline
2 & 2 & 3 & c_{j} & \\
3 & 4 & 6 & c_{j}+j &
\end{array}
\right]  ,\left[
\begin{array}
[c]{ccc||ccc}%
0 & 0 & 0 &  & 0 & 1\\
0 & 0 & 0 &  & 0 & 2\\
1 & 1 & 0 &  & 2 & 5\\
1 & 1 & 1 &  & 3 & 7\\\hline\hline
&  &  &  & r_{i} & r_{i}+i\\
&  &  &  &  &
\end{array}
\right]
\]
Note that $acs(A) =\left(  3,4,6\right)  $ and $ars(B)=(1,2,5,7)$ are sorted
in increasing order. Thus $(A,B)$ is sorted. Note that $B$ is bottom-left
flushed. Thus $(A,B)$ is flushed. Hence $(A,B)$ is sorted-flushed.
\end{example}

\begin{lemma}
\label{lem:Flushed_PC2} Let $A,B\in\left\{  0,1\right\}  ^{m\times n}$. If
$(A,B)$ is sorted-flushed, then we have%
\[
cs(A)=cs(\bar{B}) \ \Longleftrightarrow\ PC\left(  A,B\right)
\]

\end{lemma}

\begin{proof}
Assume that $(A,B)$ is sorted-flushed. We need to show $cs(A)=cs(\bar{B})
\ \Longleftrightarrow\ PC\left(  A,B\right)  $. We will show direction of
implication one by one.

\noindent\fbox{$\Rightarrow$}\ \ Assume $cs(A)=cs(\bar{B})$. Then we have
$acs(A)=acs(\bar{B})$. From Lemma \ref{lem:Flushed_PC}, we have $PC(\bar
{B},B)$. Thus we have $PC(A,B)$.

\noindent\fbox{$\Leftarrow$}\ \ Assume $PC(A,B)$. Then we have
\[
acs(A)=(1,\ldots,m+n)\setminus ars(B).
\]
From Lemma \ref{lem:Flushed_PC}, we have $PC(\bar{B},B)$. Thus we have
\[
acs(\bar{B})=(1,\ldots,m+n)\setminus ars(B).
\]

\item Thus $acs(A)=acs(\bar{B})$ and in turn $cs(A)=cs(\bar{B)}.$
\end{proof}

\begin{lemma}
\label{lem:Flushed} Let $t=\alpha^{\mu}\beta^{\nu}$ be a term. Let
$\mathcal{R}\in\left\{  0,1\right\}  ^{m\times n}.$ The following two are equivalent.

\begin{enumerate}
\item[(1)] $\mathcal{R}$ is a res-representation of $t.$

\item[(2)] $\left(  \mathcal{R},{F}_{\nu}\right)  $ is a sorted-flushed
syl-representation of $t.$
\end{enumerate}
\end{lemma}

\begin{proof}
We show each direction of implication one by one.

\noindent(1) $\Rightarrow$ (2). It follows immediately from the following claims:

\begin{description}
[leftmargin=2.7em,style=nextline,itemsep=0.0em]

\item[\textsf{C1:}] \emph{$\left(  \mathcal{R},{F}_{\nu}\right)  $ is
sorted-flushed.}

From the definition of $F_{\nu}$, it is obvious that $ars(F_{\nu})$ is sorted
in increasing order and that $F_{\nu}$ is bottom-left flushed. Thus it remains
to show that $acs(\mathcal{R})$ is sorted in increasing order. Since
$\mathcal{R}$ is a res-representation of $t$, we have $cs(\mathcal{R}) =
\bar{\nu}$. Recall that at the very beginning of this section we assumed,
without loss of generality, that $\nu$ is sorted in non-increasing order. So,
$cs(\mathcal{R})$ is sorted in non-decreasing order. Thus, $acs(\mathcal{R})$
is sorted in increasing order.

\item[\textsf{C2:}] \emph{$\left(  \mathcal{R},{F}_{\nu}\right)  $ is a
syl-representation of $t$.}

Since $\mathcal{R}$ is a res-representation of $t$, we have $rs(\mathcal{R}) =
\mu.$ From the definitions of $F_{\nu}$, we have $cs(\mathcal{F_{\nu}}) = \nu
$. Thus it remains to show that $PC(\mathcal{R},{F}_{\nu})$. Note
\[
cs(\mathcal{R})=\bar{\nu}=\overline{cs({F}_{\nu})}=cs(\overline{{F}_{\nu}}).
\]
Therefore, from \textsf{C1} and Lemma~\ref{lem:Flushed_PC2}, we have
$PC(\mathcal{R},{F}_{\nu})$.
\end{description}

\noindent(2) $\Rightarrow$ (1). Since $\left(  \mathcal{R},{F}_{\nu}\right)  $
is a syl-representation of $t$, we have $rs(\mathcal{R})=\mu$. Thus it remains
to show that $cs(\mathcal{R})=\bar{\nu}$. Since $\left(  \mathcal{R},{F}_{\nu
}\right)  $ is a sorted-flushed syl-representation of $t$, we have that
$\left(  \mathcal{R},{F}_{\nu}\right)  $ is sorted-flushed and $PC(\mathcal{R}%
,F_{\nu})$. Thus, from Lemma~\ref{lem:Flushed_PC2}, we have
\[
cs(\mathcal{R})=cs(\overline{F_{\nu}})=\overline{cs(F_{\nu})}=\bar{\nu}.
\]

\end{proof}

\begin{algorithm}
[$SylFromRes$]\label{alg:SylFromRes}\ 

\begin{description}
[leftmargin=3em,style=nextline,itemsep=0.0em]

\item[\textsf{In:}] $\mathcal{R}$, a res-representation of a term $t$

\item[\textsf{Out:}] $\left(  \mathcal{S}_{1},\mathcal{S}_{2}\right)  $, a
syl-representation of the term $t$
\end{description}

\medskip\noindent$c\leftarrow cs(\mathcal{R})$

\medskip\noindent$\left(  \mathcal{S}_{1},\mathcal{S}_{2}\right)
\leftarrow(\mathcal{R},F_{\bar{c}})$

\medskip\noindent return $\left(  \mathcal{S}_{1},\mathcal{S}_{2}\right)  $
\end{algorithm}

\begin{example}
We trace the algorithm $SylFromRes$ on the following input.

\begin{description}
[leftmargin=3em,style=nextline,itemsep=0.0em]

\item[\textsf{In:}] $\mathcal{R}=\left[
\begin{array}
[c]{cccc}%
1 & 0 & 1 & 1\\
1 & 1 & 0 & 1\\
1 & 1 & 0 & 1\\
0 & 0 & 1 & 1\\
0 & 1 & 1 & 0
\end{array}
\right]  $, \newline\newline which is a res-representation of the term
$t=\alpha_{1}^{3}\alpha_{2}^{3}\alpha_{3}^{3}\alpha_{4}^{2}\alpha_{5}^{2}%
\beta_{1}^{2}\beta_{2}^{2}\beta_{3}^{2}\beta_{4}^{1}$.
\end{description}

\medskip\noindent$c= [3,3,3,4]$

\medskip\noindent$F_{\bar{c}}= \left[
\begin{array}
[c]{cccc}%
0 & 0 & 0 & 0\\
0 & 0 & 0 & 0\\
0 & 0 & 0 & 0\\
1 & 1 & 1 & 0\\
1 & 1 & 1 & 1
\end{array}
\right]  $

\begin{description}
[leftmargin=3em,style=nextline,itemsep=0.0em]

\item[\textsf{Out:}] $\left(  \mathcal{S}_{1},\mathcal{S}_{2}\right)  =\left[
\begin{array}
[c]{cccc}%
1 & 0 & 1 & 1\\
1 & 1 & 0 & 1\\
1 & 1 & 0 & 1\\
0 & 0 & 1 & 1\\
0 & 1 & 1 & 0
\end{array}
\right]  , \left[
\begin{array}
[c]{cccc}%
0 & 0 & 0 & 0\\
0 & 0 & 0 & 0\\
0 & 0 & 0 & 0\\
1 & 1 & 1 & 0\\
1 & 1 & 1 & 1
\end{array}
\right]  $, \newline\newline which is a syl-representation of the term $t$.
\end{description}
\end{example}

\begin{lemma}
\label{SylFromRes} The algorithm~\ref{alg:SylFromRes} ($SylFromRes$) is
correct. Thus if a term has a res-representation then it has a syl-representation.
\end{lemma}

\begin{proof}
Let $\mathcal{R}$ be an input, a res-representation of a term $t=\alpha^{\mu
}\beta^{\nu}.$ Then $c=\bar{\nu}$ and thus $\bar{c}=\nu.$ From
Lemma~\ref{lem:Flushed}, $\left(  \mathcal{R},F_{\bar{c}}\right)  $ is a
(sorted-flushed) syl-representation of the term $t.$
\end{proof}

\subsection{If a term has a syl-representation then it has a
res-representation.}

The proof is constructive, that is, it provides an algorithm that takes a
syl-representation and produces a res-representation
(Algorithm~\ref{alg:ResFromSyl}). We will again use the key lemma
(Lemma~\ref{lem:Flushed} from the previous subsection) that establishes a
crucial relationship between the two representations (syl and res). In order
to use the key lemma, we need to find an algorithm that transforms a given
syl-representation into a sorted-flushed syl-representation. We will describe
such an algorithm in this subsection. We will divide, naturally, the algorithm
into two subalgorithms.

\begin{itemize}
\item Algorithm~\ref{alg:Sort} ($Sort$):

It transforms a syl-representation of a term into a sorted syl-representation
of the term. It essentially carries out bubble sort.

\item Algorithm~\ref{alg:Flush} ($Flush$):

It transforms a sorted syl-representation of a term into a sorted-flushed
syl-representation of the term. It essentially carries out repeated swapping
of entries of the syl-representation to make it flushed while remaining sorted syl-representation.
\end{itemize}

\noindent Most of this subsection will devoted in describing and proving the correctness
of the two subalgorithms. 
Now we plunge into details.

\newpage
\begin{algorithm}
[$Sort$]\label{alg:Sort}\ 

\begin{description}
[leftmargin=3em,style=nextline,itemsep=0.0em]

\item[\textsf{In:}] $\left(  \mathcal{S}_{1},\mathcal{S}_{2}\right)  $, a
syl-representation of a term $t$

\item[\textsf{Out:}] $\left(  \mathcal{S}_{1}^{\prime},\mathcal{S}_{2}%
^{\prime}\right)  $, a sorted syl-representation of the term $t$
\end{description}

\begin{enumerate}
\item $\left(  \mathcal{S}_{1}^{\prime},\mathcal{S}_{2}^{\prime}\right)
\leftarrow\left(  \mathcal{S}_{1},\mathcal{S}_{2}\right)  $

\item \label{Sort1:Repeat} Repeat

\begin{enumerate}
\item \label{Sort1:C} $C\leftarrow acs\left(  \mathcal{S}_{1}^{\prime}\right)
$

\item \label{Sort1:check} If $C$ is in increasing order then exit the Repeat loop

\item \label{Sort1:j} Find $j\in\{1,\ldots,n-1\}$ such that $C_{j}>C_{j+1}$

\item \label{Sort1:h} $h\leftarrow C_{j}-C_{j+1}$

\item \label{Sort1:i} Repeat $h$ times

\begin{enumerate}
\item Find $i\in\{1,\ldots,m\}$ such that $\mathcal{S}_{1,i,j}^{\prime}=1$ and
$\mathcal{S}_{1,i,j+1}^{\prime}=0$

\item Swap $\mathcal{S}_{1,i,j}^{\prime}$ and $\mathcal{S}_{1,i,j+1}^{\prime}$
\end{enumerate}
\end{enumerate}

\item \label{Sort2:Repeat} Repeat

\begin{enumerate}
\item \label{Sort2:R} $R\leftarrow ars\left(  \mathcal{S}_{2}^{\prime}\right)
$

\item \label{Sort2:check} If $R$ is in increasing order then exit the Repeat loop

\item \label{Sort2:i} Find $i\in\{1,\ldots,m-1\}$ such that $R_{i}>R_{i+1}$

\item \label{Sort2:h} $h\leftarrow R_{i}-R_{i+1}$

\item \label{Sort2:j} Repeat $h$ times

\begin{enumerate}
\item Find $j\in\{1,\ldots,n\}$ such that $\mathcal{S}_{2,i,j}^{\prime}=1$ and
$\mathcal{S}_{2,\,i+1,\,j}^{\prime}=0$

\item Swap $\mathcal{S}_{2,i,j}$ and $\mathcal{S}_{2,i+1,\,j}$
\end{enumerate}
\end{enumerate}

\item Return $\left(  \mathcal{S}_{1}^{\prime},\mathcal{S}_{2}^{\prime
}\right)  $
\end{enumerate}
\end{algorithm}

\begin{example}
We trace the algorithm $Sort$ on the following input.

\begin{description}
[leftmargin=3em,style=nextline,itemsep=0.0em]

\item[\textsf{In:}] $\left(  \mathcal{S}_{1},\mathcal{S}_{2}\right)  =\left[
\begin{array}
[c]{cccc}%
0 & 1 & 1 & 1\\
0 & 1 & 1 & 1\\
1 & 1 & 1 & 0\\
0 & 0 & 1 & 1\\
0 & 1 & 1 & 0
\end{array}
\right]  ,\left[
\begin{array}
[c]{cccc}%
0 & 0 & 0 & 0\\
0 & 1 & 0 & 0\\
1 & 0 & 1 & 0\\
0 & 0 & 0 & 0\\
1 & 1 & 1 & 1
\end{array}
\right]  $, \newline\newline which\ is a syl-representation of the term
$t=\alpha_{1}^{3}\alpha_{2}^{3}\alpha_{3}^{3}\alpha_{4}^{2}\alpha_{5}^{2}%
\beta_{1}^{2}\beta_{2}^{2}\beta_{3}^{2}\beta_{4}^{1}$.
\end{description}

\begin{enumerate}
\item $\left(  \mathcal{S}_{1}^{\prime},\mathcal{S}_{2}^{\prime}\right)
=\left(  \mathcal{S}_{1},\mathcal{S}_{2}\right)  $

\item \medskip\noindent Iteration 1

\begin{enumerate}
\item $C=[2,6,8,7]$

\item $C$ is not sorted

\item $j=3$

\item $h=1$

\begin{enumerate}
\item $i=3$

\item Swap $S^{\prime}_{1,3,3}$ and $S^{\prime}_{1,3,4}$

$\mathcal{S}_{1}^{\prime}=\left[
\begin{array}
[c]{cccc}%
0 & 1 & 1 & 1\\
0 & 1 & 1 & 1\\
1 & 1 & 0 & 1\\
0 & 0 & 1 & 1\\
0 & 1 & 1 & 0
\end{array}
\right]  $
\end{enumerate}
\end{enumerate}

\medskip\noindent Iteration 2

\begin{enumerate}
\item $C=[2,6,7,8]$

\item $C$ is sorted in increasing order
\end{enumerate}

\item Iteration 1

\begin{enumerate}
\item $R=[1,3,5,4,9]$

\item $R$ is not sorted

\item $i=3$

\item $h=1$

\begin{enumerate}
\item $j=1$

\item Swap $S^{\prime}_{2,3,1}$ and $S^{\prime}_{2,4,1}$

$\mathcal{S}^{\prime}_{2} = \left[
\begin{array}
[c]{cccc}%
0 & 0 & 0 & 0\\
0 & 1 & 0 & 0\\
0 & 0 & 1 & 0\\
1 & 0 & 0 & 0\\
1 & 1 & 1 & 1
\end{array}
\right]  $
\end{enumerate}
\end{enumerate}

\medskip\noindent Iteration 2

\begin{enumerate}
\item $R=[1,3,4,5,9]$

\item $R$ is sorted in increasing order
\end{enumerate}

\item Return $\left(  \mathcal{S}_{1}^{\prime},\mathcal{S}_{2}^{\prime
}\right)  $
\end{enumerate}

\begin{description}
\item[\textsf{Out:}] $\left(  \mathcal{S}_{1}^{\prime},\mathcal{S}_{2}%
^{\prime}\right)  = \left[
\begin{array}
[c]{cccc}%
0 & 1 & 1 & 1\\
0 & 1 & 1 & 1\\
1 & 1 & 0 & 1\\
0 & 0 & 1 & 1\\
0 & 1 & 1 & 0
\end{array}
\right]  , \left[
\begin{array}
[c]{cccc}%
0 & 0 & 0 & 0\\
0 & 1 & 0 & 0\\
0 & 0 & 1 & 0\\
1 & 0 & 0 & 0\\
1 & 1 & 1 & 1
\end{array}
\right]  $, \newline\newline which\ is a sorted syl-representation of the term
$t.$
\end{description}
\end{example}

\begin{lemma}
\label{lem:Sort} The algorithm~\ref{alg:Sort} ($Sort$) is correct.
\end{lemma}

\begin{proof}
Let $\left(  \mathcal{S}_{1},\mathcal{S}_{2}\right)  $ be an input, that is, a
syl-representation of a term $t=\alpha^{\mu}\beta^{\nu}$. The correctness of
the algorithm is immediate from the following claims.

\begin{description}
[leftmargin=2.7em,style=nextline,itemsep=0.5em]

\item[\textsf{C1:}] \emph{Right after Step~\ref{Sort1:Repeat}, $\left(
\mathcal{S}_{1}^{\prime},\mathcal{S}_{2}^{\prime}\right)  $ is a
syl-representation of the term $t$ and $acs( \mathcal{S}_{1}^{\prime})$ is
sorted in increasing order.} The proof of the claim is immediate from the
following sub-claims.

\begin{enumerate}
\item \emph{Right before Step~\ref{Sort1:C}, $\left(  \mathcal{S}_{1}^{\prime
},\mathcal{S}_{2}^{\prime}\right)  $ is a syl-representation of the term $t$}.

We prove it by induction on the number of iterations. At the first iteration,
it is trivially true since $\left(  \mathcal{S}_{1}^{\prime},\mathcal{S}%
_{2}^{\prime}\right)  =\left(  \mathcal{S}_{1},\mathcal{S}_{2}\right)  .$ We
assume that it is true after some number of iterations. We need to show that
it is still true after one more iteration.

It is immediate from the following observations.

\begin{itemize}
\item $rs\left(  \mathcal{S}_{1}^{\prime}\right)  =\mu$.

Obvious since the loop body does not change $rs\left(  \mathcal{S}_{1}%
^{\prime}\right)  .$

\item $cs\left(  \mathcal{S}_{2}^{\prime}\right)  =\nu$.

Obvious since the loop body does not change $\mathcal{S}_{2}^{\prime}$

\item $PC\left(  \mathcal{S}_{1}^{\prime},\mathcal{S}_{2}^{\prime}\right)  .$

Let $A,B\in\{1,\ldots,m\}$ such that $A=C_{j}$ and $B=C_{j+1}$. From
Step~\ref{Sort1:h} we know that $C_{j}=C_{j+1}+h$. Inside Step~\ref{Sort1:i},
we have increased $C_{j+1}$ by $h$ and decreased $C_{j}$ by $h$ . Thus after
Step~\ref{Sort1:i} we have
\begin{align*}
C_{j+1}  &  \leftarrow C_{j+1}+h\\
C_{j}  &  \leftarrow C_{j}-h
\end{align*}
Hence $C_{j}=B$ and $C_{j+1}=A$, that is, we have swapped $C_{j}$ and
$C_{j+1}.$ In other words, the loop body does not change $C$ as a set. Thus
$PC\left(  \mathcal{S}_{1}^{\prime},\mathcal{S}_{2}^{\prime}\right)  $ still holds.
\end{itemize}

\item \emph{The repeat loop in Step~\ref{Sort1:Repeat} terminates}.

It is a bubble sort algorithm executed on a finite list $C$. Therefore it terminates.

\item \emph{In Step~\ref{Sort1:j}, there exists $j\in\{1,\ldots,n-1\}$ such
that $C_{j}>C_{j+1}$}.

The claim is immediate from the following observations.

\begin{itemize}
\item Since we are at Step~\ref{Sort1:j}, the `if' condition in
Step~\ref{Sort1:check} is not satisfied. Hence $C$ is not in increasing order.
Thus there exists $j\in\{1,\ldots,n-1\}$ such that $C_{j}\geq C_{j+1}$.

\item Since $\left(  \mathcal{S}_{1}^{\prime},\mathcal{S}_{2}^{\prime}\right)
$ is a syl-representation, we have $C_{j}\neq C_{j+1}$.
\end{itemize}

\item \emph{In Step~\ref{Sort1:i}(i), there exists $i\in\{1,\ldots,m\}$ such
that $\mathcal{S}_{1,i,j}^{\prime}=1$ and $\mathcal{S}_{1,i,j+1}^{\prime}=0$}.

From Step~\ref{Sort1:h} we know $h=C_{j}-C_{j+1}$. Hence there must exist $h$
different $i\in\{1,\ldots,m\}$ such that $\mathcal{S}_{1,i,j}^{\prime}=1$ and
$\mathcal{S}_{1,i,j+1}^{\prime}=0$.
\end{enumerate}

\item[\textsf{C2:}] \emph{Right after Step~\ref{Sort2:Repeat}, $\left(
\mathcal{S}_{1}^{\prime},\mathcal{S}_{2}^{\prime}\right)  $ is a
syl-representation of the term $t$ and $acs(( \mathcal{S}_{1}^{\prime})$ and
$ars(\mathcal{S}_{2}^{\prime})$ are sorted in increasing order.} The proof of
the claim is symmetric to the proof of \text{C1}. One only needs to switch the
roles of $S_{1}^{\prime}$ and $S_{2}^{\prime}$ and the roles of columns and rows.
\end{description}
\end{proof}

\begin{algorithm}
[$Flush$]\label{alg:Flush}\ 

\begin{description}
[leftmargin=3em,style=nextline,itemsep=0.0em]

\item[\textsf{In:}] $\left(  \mathcal{S}_{1},\mathcal{S}_{2}\right)  $, a
sorted syl-representation of a term $t$

\item[\textsf{Out:}] $\left(  \mathcal{S}_{1}^{\prime},\mathcal{S}_{2}%
^{\prime}\right)  $, a sorted-flushed syl-representation of the term $t$
\end{description}

\begin{enumerate}
\item $\left(  \mathcal{S}_{1}^{\prime},\mathcal{S}_{2}^{\prime}\right)
\leftarrow\left(  \mathcal{S}_{1},\mathcal{S}_{2}\right)  $

\item Repeat

\begin{enumerate}
\item \label{Flush:Begin} If $\left(  \mathcal{S}_{1}^{\prime},\mathcal{S}%
_{2}^{\prime}\right)  $ is flushed then return $\left(  \mathcal{S}%
_{1}^{\prime},\mathcal{S}_{2}^{\prime}\right)  $

\item $c\leftarrow cs(\mathcal{S}_{1}^{\prime})$

$r\leftarrow rs(\mathcal{S}_{2}^{\prime})$

$C\leftarrow acs\left(  \mathcal{S}_{1}^{\prime}\right)  $

$R\leftarrow ars\left(  \mathcal{S}_{2}^{\prime}\right)  $

\item \label{Flush:Findij} Find $\left(  i,j\right)  \in\{1,\ldots
,m-1\}\times\left\{  1,\ldots,n\right\}  \ $such that $\mathcal{S}%
_{2,i,j}^{\prime}=1\ $and\ $\mathcal{S}_{2,i+1,j}^{\prime}=0$

Swap $\mathcal{S}_{2,i,j}^{\prime}\ $and $\mathcal{S}_{2,i+1,j}^{\prime}$


\item \label{Flush:i1} $i_{\ell} \leftarrow\mathrm{min}\,\{ \, k ~|~
r_{k}=r_{i}, \, k \leq i \}$

\item \label{Flush:Findj1} If $i_{\ell}< i$ then

\quad Find $j\in\{1,\ldots,n\}$ such that $\mathcal{S}_{2,i_{\ell},j}^{\prime
}=1$ and $\mathcal{S}_{2,i,j}^{\prime}=0$

\quad Swap $\mathcal{S}_{2,i_{\ell},j}^{\prime}$ and $\mathcal{S}%
_{2,i,j}^{\prime}$

\item \label{Flush:i2} $i_{u} \leftarrow\mathrm{max}\,\{ \, k ~|~
r_{k}=r_{i+1}, \, k \ge i+1 \}$

\item \label{Flush:Findj2} If $i+1 < i_{u}$ then

\quad Find $j\in\{1,\ldots,n\}$ such that $\mathcal{S}_{2,i+1,j}^{\prime}=1$
and $\mathcal{S}_{2,i_{u},j}^{\prime}=0$.

\quad Swap $\mathcal{S}_{2,i+1,j}^{\prime}$ and $\mathcal{S}_{2,i_{u}%
,j}^{\prime}$


\item \label{Flush:Findi}Find $i\in\{1,\ldots,m\}$ and $j_{\ell}<j_{u}%
\in\{1,\ldots,n\}$ such that

\quad$\mathcal{S}_{1,i,j_{\ell}}^{\prime}=0$ and $\mathcal{S}_{1,i,j_{u}%
}^{\prime}=1$ and $C_{j_{\ell}}=R_{i_{\ell}}-1$ and $C_{j_{u}}=R_{i_{u}}+1$

Swap $\mathcal{S}_{1,i,j_{\ell}}^{\prime}$ and $\mathcal{S}_{1,i,j_{u}%
}^{\prime}$
\end{enumerate}
\end{enumerate}
\end{algorithm}

\begin{example}
We trace the algorithm $Flush$ on the following input.

\begin{description}
[leftmargin=3em,style=nextline,itemsep=0.0em]

\item[\textsf{In:}] $\left(  \mathcal{S}_{1},\mathcal{S}_{2}\right)  =\left[
\begin{array}
[c]{cccc}%
0 & 1 & 1 & 1\\
0 & 1 & 1 & 1\\
1 & 1 & 0 & 1\\
0 & 0 & 1 & 1\\
0 & 1 & 1 & 0
\end{array}
\right]  ,\left[
\begin{array}
[c]{cccc}%
0 & 0 & 0 & 0\\
0 & 1 & 0 & 0\\
0 & 0 & 1 & 0\\
1 & 0 & 0 & 0\\
1 & 1 & 1 & 1
\end{array}
\right]  $, \newline\newline which is a sorted syl-representation of the term
$t=\alpha_{1}^{3}\alpha_{2}^{3}\alpha_{3}^{3}\alpha_{4}^{2}\alpha_{5}^{2}%
\beta_{1}^{2}\beta_{2}^{2}\beta_{3}^{2}\beta_{4}^{1}$.
\end{description}

\begin{enumerate}
\item $\left(  \mathcal{S}_{1}^{\prime},\mathcal{S}_{2}^{\prime}\right)
=\left(  \mathcal{S}_{1},\mathcal{S}_{2}\right)  $

\item Iteration 1

\begin{enumerate}
\item $\left(  \mathcal{S}_{1}^{\prime},\mathcal{S}_{2}^{\prime}\right)  $ is
not flushed

\item $c=[1,4,4,4]$

$r=[0,1,1,1,4]$

$C=[2,6,7,8]$

$R=[1,3,4,5,9]$

\item $i=2$, $j=2$

Swap $\mathcal{S}_{2,2,2}^{\prime}$ and $\mathcal{S}_{2,3,2}^{\prime}$

$\left(  \mathcal{S}_{1}^{\prime},\mathcal{S}_{2}^{\prime}\right)  =\left[
\begin{array}
[c]{cccc}%
0 & 1 & 1 & 1\\
0 & 1 & 1 & 1\\
1 & 1 & 0 & 1\\
0 & 0 & 1 & 1\\
0 & 1 & 1 & 0
\end{array}
\right]  ,\left[
\begin{array}
[c]{cccc}%
0 & 0 & 0 & 0\\
0 & 0 & 0 & 0\\
0 & 1 & 1 & 0\\
1 & 0 & 0 & 0\\
1 & 1 & 1 & 1
\end{array}
\right]  $

\item $i_{\ell}=2$

\item $i_{\ell} \not <  i$

\item $i_{u}=4$

\item $i+1 < i_{u}$

$j=2$

Swap $\mathcal{S}_{2,3,2}^{\prime}$ and $\mathcal{S}_{2,4,2}^{\prime}$

$\left(  \mathcal{S}_{1}^{\prime},\mathcal{S}_{2}^{\prime}\right)  =\left[
\begin{array}
[c]{cccc}%
0 & 1 & 1 & 1\\
0 & 1 & 1 & 1\\
1 & 1 & 0 & 1\\
0 & 0 & 1 & 1\\
0 & 1 & 1 & 0
\end{array}
\right]  ,\left[
\begin{array}
[c]{cccc}%
0 & 0 & 0 & 0\\
0 & 0 & 0 & 0\\
0 & 0 & 1 & 0\\
1 & 1 & 0 & 0\\
1 & 1 & 1 & 1
\end{array}
\right]  $




\item $i=1$ and $j_{\ell}=1$, $j_{u}=2$

Swap $\mathcal{S}_{1,1,1}^{\prime}$ and $\mathcal{S}_{1,1,2}^{\prime}$

$\left(  \mathcal{S}_{1}^{\prime},\mathcal{S}_{2}^{\prime}\right)  =\left[
\begin{array}
[c]{cccc}%
1 & 0 & 1 & 1\\
0 & 1 & 1 & 1\\
1 & 1 & 0 & 1\\
0 & 0 & 1 & 1\\
0 & 1 & 1 & 0
\end{array}
\right]  , \left[
\begin{array}
[c]{cccc}%
0 & 0 & 0 & 0\\
0 & 0 & 0 & 0\\
0 & 0 & 1 & 0\\
1 & 1 & 0 & 0\\
1 & 1 & 1 & 1
\end{array}
\right]  $
\end{enumerate}

\medskip\noindent Iteration 2

\begin{enumerate}
\item $\left(  \mathcal{S}_{1}^{\prime},\mathcal{S}_{2}^{\prime}\right)  $ is
not flushed

\item $c=[2,3,4,4]$

$r=[0,0,1,2,4]$

$C=[3,5,7,8]$

$R=[1,2,4,6,9]$

\item $i=3$, $j=3$

Swap $\mathcal{S}_{2,3,3}^{\prime}$ and $\mathcal{S}_{2,4,3}^{\prime}$

$\left(  \mathcal{S}_{1}^{\prime},\mathcal{S}_{2}^{\prime}\right)  =\left[
\begin{array}
[c]{cccc}%
1 & 0 & 1 & 1\\
0 & 1 & 1 & 1\\
1 & 1 & 0 & 1\\
0 & 0 & 1 & 1\\
0 & 1 & 1 & 0
\end{array}
\right]  ,\left[
\begin{array}
[c]{cccc}%
0 & 0 & 0 & 0\\
0 & 0 & 0 & 0\\
0 & 0 & 0 & 0\\
1 & 1 & 1 & 0\\
1 & 1 & 1 & 1
\end{array}
\right]  $

\item $i_{\ell}=3$

\item $i_{\ell} \not <  i$

\item $i_{u}=4$

\item $i+1 \not <  i_{u}$




\item $i=2$ and $j_{\ell}=1$, $j_{u}=3$

Swap $\mathcal{S}_{1,2,1}^{\prime}$ and $\mathcal{S}_{1,2,3}^{\prime}$

$\left(  \mathcal{S}_{1}^{\prime},\mathcal{S}_{2}^{\prime}\right)  =\left[
\begin{array}
[c]{cccc}%
1 & 0 & 1 & 1\\
1 & 1 & 0 & 1\\
1 & 1 & 0 & 1\\
0 & 0 & 1 & 1\\
0 & 1 & 1 & 0
\end{array}
\right]  , \left[
\begin{array}
[c]{cccc}%
0 & 0 & 0 & 0\\
0 & 0 & 0 & 0\\
0 & 0 & 0 & 0\\
1 & 1 & 1 & 0\\
1 & 1 & 1 & 1
\end{array}
\right]  $
\end{enumerate}

\medskip\noindent Iteration 3

\begin{enumerate}
\item $\left(  \mathcal{S}_{1}^{\prime},\mathcal{S}_{2}^{\prime}\right)  $ is flushed
\end{enumerate}
\end{enumerate}

\begin{description}
\item[\textsf{Out:}] $\left(  \mathcal{S}_{1}^{\prime},\mathcal{S}_{2}%
^{\prime}\right)  =\left[
\begin{array}
[c]{cccc}%
1 & 0 & 1 & 1\\
1 & 1 & 0 & 1\\
1 & 1 & 0 & 1\\
0 & 0 & 1 & 1\\
0 & 1 & 1 & 0
\end{array}
\right]  , \left[
\begin{array}
[c]{cccc}%
0 & 0 & 0 & 0\\
0 & 0 & 0 & 0\\
0 & 0 & 0 & 0\\
1 & 1 & 1 & 0\\
1 & 1 & 1 & 1
\end{array}
\right]  $, \newline\newline which is a sorted-flushed syl-representation of
the term $t$.
\end{description}
\end{example}

\begin{lemma}
\label{lem:Flush} The algorithm~\ref{alg:Flush} $(Flush)$ is correct.
\end{lemma}

\begin{proof}
Let $\left(  \mathcal{S}_{1},\mathcal{S}_{2}\right)  $ be an input, that is, a
sorted syl-representation of a term $t=\alpha^{\mu}\beta^{\nu}$. The
correctness of the algorithm is immediate from the following claims.

\begin{description}
[leftmargin=2.7em,style=nextline,itemsep=0.5em]

\item[\textsf{C1:}] \emph{Right before Step~\ref{Flush:Begin}, $\left(
\mathcal{S}_{1}^{\prime},\mathcal{S}_{2}^{\prime}\right)  $ is a sorted
syl-representation of the term $t$.}

We prove it by induction on the number of iterations. At the first iteration,
it is trivially true since $\left(  \mathcal{S}_{1}^{\prime},\mathcal{S}%
_{2}^{\prime}\right)  =\left(  \mathcal{S}_{1},\mathcal{S}_{2}\right)  .$ We
assume that it is true after some number of iterations. We need to show that
it is still true after one more iteration. It
is immediate from the following two sub-claims,
where $C^{\prime}=acs\left(  \mathcal{S}_{1}^{\prime
}\right)  $ and $R^{\prime}=ars\left(  \mathcal{S}_{2}^{\prime}\right)$
at the end of
Step~\ref{Flush:Findi}.  

\begin{itemize}
\item \emph{$C^{\prime}$ and $R^{\prime}$ are sorted in increasing order}.

Note that Steps~\ref{Flush:Findij}--\ref{Flush:Findj2} transform $\mathcal{S}_2'$  
by carrying out swaps along columns, as depicted by the following diagram.  

\begin{tikzpicture}[line cap=round,x=1.0cm,y=1.0cm]
\clip(-1.14,-0.162) rectangle (14.59,4.2);
\draw (0.,4.)-- (5.,4.);
\draw (5.,4.)-- (5.,0.);
\draw (0.,4.)-- (0.,0.);
\draw (0.,0.)-- (5.,0.);
\draw (8.,4.)-- (13.,4.);
\draw (8.,4.)-- (8.,0.);
\draw (8.,0.)-- (13.,0.);
\draw (13.,0.)-- (13.,4.);
\draw (5.5,2.3) node[anchor=north west] {$\Longrightarrow$};
\draw (-0.900,3.5) node[anchor=north west] {$i_\ell$};
\draw (7.065, 3.5) node[anchor=north west] {$i_\ell$};
\draw (-0.900,1) node[anchor=north west] {$i_u$};
\draw (7.065, 1) node[anchor=north west] {$i_u$};
\draw (-0.900,2.5) node[anchor=north west] {$i$};
\draw (7.065, 2.5) node[anchor=north west] {$i$};
\draw (-0.900,2.0) node[anchor=north west] {$i+1$};
\draw (7.065, 2.0) node[anchor=north west] {$i+1$};
\draw (1.1,   2.5) node[anchor=north west] {$1$};
\draw (9.065, 2.5) node[anchor=north west] {$0$};
\draw (1.1,   2.0) node[anchor=north west] {$0$};
\draw (9.065, 2.0) node[anchor=north west] {$1$};
\draw (3.1,   2.5) node[anchor=north west] {$0$};
\draw (11.065, 2.5) node[anchor=north west] {$1$};
\draw (4.1,   2.0) node[anchor=north west] {$1$};
\draw (12.065, 2.0) node[anchor=north west] {$0$};
\draw (3.1,   3.5) node[anchor=north west] {$1$};
\draw (11.065,3.5) node[anchor=north west] {$0$};
\draw (4.1,   1.0) node[anchor=north west] {$0$};
\draw (12.065,1.0) node[anchor=north west] {$1$};

\end{tikzpicture}

Thus\begin{align*}
R_{t}^{\prime}  &  = \left\{
\begin{array}
[c]{lll}%
R_{t}-1 & \text{if} & t=i_{\ell}\\
R_{t}+1 & \text{if} & t=i_{u}\\
R_{t} & \text{else} &
\end{array}
\right.
\end{align*}
From Step~\ref{Flush:Findi},   it is immediate that
\begin{align*}
C_{t}^{\prime}  &  = \left\{
\begin{array}
[c]{lll}%
C_{t}+1 & \text{if} & t=j_{\ell}\\
C_{t}-1 & \text{if} & t=j_{u}\\
C_{t} & \text{else} &
\end{array}
\right.
\end{align*}
Next recall that $R$ and $C$ are sorted in increasing order. Thus we only need
to show that $R_{i_{\ell}}^{\prime}>R_{i_{\ell}-1}$,\; $R_{i_{u}}^{\prime
}<R_{i_{u}+1}$,\; $C_{j_{\ell}}^{\prime}<C_{j_{\ell}+1}$\; and\; $C_{j_{u}%
}^{\prime}>C_{j_{u}-1}$. We show them one by one.

\begin{itemize}
\item $R_{i_{\ell}}^{\prime}>R_{i_{\ell}-1}$

Recall that $R_{i_{\ell}}^{\prime}=R_{i_{\ell}}-1.$ Since $r_{i_{\ell}%
}>r_{i_{\ell}-1}$ we have $R_{i_{\ell}}-1>R_{i_{\ell}-1}.\ $Thus $R_{i_{\ell}%
}^{\prime}>R_{i_{\ell}-1}.$

\item $R_{i_{u}}^{\prime}<R_{i_{u}+1}$

Recall that $R_{i_{u}}^{\prime}=R_{i_{u}}+1.$ Since $r_{i_{u}}<r_{i_{u}+1}$ we
have $R_{i_{u}}+1<R_{i_{u}+1}.\ $Thus $R_{iu}^{\prime}<R_{i_{u}+1}.$

\item $C_{j_{\ell}}^{\prime}<C_{j_{\ell}+1}$

Recall that $C_{j_{\ell}}^{\prime}=C_{j_{\ell}}+1$. Since $C_{j_{\ell}%
}=R_{i_{\ell}}-1$ we have $C_{j_{\ell}}^{\prime}=R_{i_{\ell}}$. Since
$R_{i_{\ell}}$ appears~in~$R$ and $PC\left(  C,R\right)  ,$ we see that
$R_{i_{\ell}}$ does not appear in $C.$  Hence $R_{i_{\ell}}<C_{j_{\ell}+1}.$
Thus $C_{j_{\ell}}^{\prime}<C_{j_{\ell}+1}.$

\item $C_{j_{u}}^{\prime}>C_{j_{u}-1}$

Recall that $C_{j_{u}}^{\prime}=C_{j_{u}}-1$. Since $C_{j_{u}}=R_{i_{u}}+1$ we
have $C_{j_{u}}^{\prime}=R_{i_{u}}$. Since $R_{i_{u}}$ appears~in~$R$ and
$PC\left(  C,R\right)  ,$ we see that $R_{i_{u}}$ does not appear in $C.$
Hence $R_{i_{u}}>C_{j_{u}-1}.$ Thus $C_{j_{u}}^{\prime}>C_{j_{u}-1}.$
\end{itemize}

\item \emph{$\left(  \mathcal{S}_{1}^{\prime},\mathcal{S}_{2}^{\prime}\right)
$ is a syl-representation of the term $t$ at the end of Step~\ref{Flush:Findi}%
}.

\begin{itemize}
\item $rs\left(  \mathcal{S}_{1}^{\prime}\right)  =\mu$.

Obvious since the loop body does not change $rs\left(  \mathcal{S}_{1}%
^{\prime}\right)  .$

\item $cs\left(  \mathcal{S}_{2}^{\prime}\right)  =\nu$.

Obvious since the loop body does not change $cs\left(  \mathcal{S}_{2}%
^{\prime}\right)  .$

\item $PC\left(  \mathcal{S}_{1}^{\prime},\mathcal{S}_{2}^{\prime}\right)  $ holds.

Note $R_{i_{\ell}}^{\prime}=C_{j_{\ell}}$, $R_{i_{u}}^{\prime}=C_{j_{u}}$ and
$C_{j_{\ell}}^{\prime}=R_{i_{\ell}}$, $C_{j_{u}}^{\prime}=R_{i_{u}}$. Note
that the others do not change.

In other words $C^{\prime}\cup R^{\prime}=C\cup R$ as sets. Thus $PC\left(
\mathcal{S}_{1}^{\prime},\mathcal{S}_{2}^{\prime}\right)  $ holds.
\end{itemize}
\end{itemize}

\item[\textsf{C2:}] \emph{The main loop (Repeat) terminates.}

For every iteration of the main loop, at least one \textquotedblleft out of
order\textquotedblright\ pair of $(1,0)$ in $\mathcal{S}_{2}$ is swapped. Thus
the algorithm terminates.

\item[\textsf{C3:}] \emph{In Step~\ref{Flush:Findij}, there exists }$\left(
\emph{i,j}\right)  $\emph{ satisfying the conditions stated in the step.}

Since we are at Step~\ref{Flush:Findij}, the `if' condition in
Step~\ref{Flush:Begin} is not satisfied. Hence $\left(  \mathcal{S}%
_{1}^{\prime},\mathcal{S}_{2}^{\prime}\right)  $ is not flushed. Thus
there exists $\left(  i,j\right)  \in\{1,\ldots,m-1\}\times\left\{
1,\ldots,n\right\}  \ $such that $\mathcal{S}_{2,i,j}^{\prime}=1\ $%
and\ $\mathcal{S}_{2,i+1,j}^{\prime}=0$.

\item[\textsf{C4:}] \emph{In Step~\ref{Flush:Findj1}, there exists j
satisfying the conditions stated in the step.}

Let $r^{\prime}=rs(\mathcal{S}_{2}^{\prime})$ right before entering the step. Note

\begin{itemize}
\item From Step~\ref{Flush:Findij}, we have $r_{i}^{\prime}=r_{i}-1$.

\item From Step~\ref{Flush:i1}, we have $r_{i_{\ell}}^{\prime}=r_{i_\ell}=r_{i}$.
\end{itemize}

Therefore there exist $j\in\{1,\ldots,n\}$ such that $\mathcal{S}%
_{2,i,j}^{\prime}=0$ and $\mathcal{S}_{2,i_{\ell},j}^{\prime}=1$.

\item[\textsf{C5:}] \emph{In Step~\ref{Flush:Findj2}, there exists j
satisfying the conditions stated in the step.}

Let $r^{\prime}=rs(\mathcal{S}_{2}^{\prime})$ right before entering the step. Note\ 

\begin{itemize}
\item From Step~\ref{Flush:Findij}, we have $r_{i+1}^{\prime}=r_{i+1}+1$.

\item From Step \ref{Flush:i2}, we have $r_{i_{u}}^{\prime}=r_{i_{u}}=r_{i+1}$.
\end{itemize}

Therefore there exist $j\in\{1,\ldots,n\}$ such that $\mathcal{S}%
_{2,i+1,j}^{\prime}=1$ and $\mathcal{S}_{2,i_{u},j}^{\prime}=0$.

\item[\textsf{C6:}] \emph{In Step~\ref{Flush:Findi}, there exist $i,j_{\ell
},j_{u}$ satisfying the conditions stated in the step.}

\begin{enumerate}
\item \label{lem:MakeNiceC41} From Step~\ref{Flush:i1}, we have

\begin{itemize}
\item if $i_{\ell}=1$ then $R_{i_{\ell}}\geq2$ ~~(since $\mathcal{S}%
_{2,i,j}^{\prime}=1$ and in turn $r_{i_{\ell}}=r_{i}\geq1$).

\item if $i_{\ell}>1$ then $R_{i_{\ell}}-R_{i_{\ell}-1}\geq2$ ~~(since
$r_{i_{\ell}}\geq r_{i_{\ell}-1}+1$).
\end{itemize}

Thus $R_{i_{\ell}}-1$ does not appear in $R$. Hence it must appear in $C$.
Thus there exists $j_{\ell}$ such that $C_{j_{\ell}}=R_{i_{\ell}}-1$.

\item \label{lem:MakeNiceC42} From Step~\ref{Flush:i2}, we have

\begin{itemize}
\item if $i_{u}=m$ then $R_{i_{u}}\leq m+n-1$ ~~(since $\mathcal{S}%
_{2,i+1,j}^{\prime}=0$ and in turn $r_{i_{u}}=r_{i+1}\leq n-1$).

\item if $i_{u}<m$ then $R_{i_{u}+1}-R_{i_{u}}\geq2$ ~~(since $r_{i_{u}}+1\leq
r_{i_{u}+1}$).
\end{itemize}

Thus $R_{i_{u}}+1$ does not appear in $R$. Hence it must appear in $C$. Thus
there exists $j_{u}$ such that $C_{j_{u}}=R_{i_{u}}+1$.
\end{enumerate}

Therefore there exist $j_{\ell},j_{u}\in\{1,\ldots,n\}$ such that $C_{j_{\ell
}}=R_{i_{\ell}}-1$ and $C_{j_{u}}=R_{i_{u}}+1$.

It remains to show that there exists $i$ that satisfies the conditions of
Step~\ref{Flush:Findi}.

Note that $R_{i_{\ell}},R_{i_{\ell}+1},\ldots,R_{i_{u}}$ appear in $R.$ Hence
they do not appear in $C.$ Note that%
\begin{align*}
R_{i_{\ell}}  &  =r_{i}+i_{\ell}\\
R_{i_{\ell}+1}  &  =r_{i}+i_{\ell}+1\\
&  \vdots\\
R_{i}  &  =r_{i}+i\\
R_{i+1}  &  =r_{i+1}+i+1\\
R_{i+2}  &  =r_{i+1}+i+2\\
&  \vdots\\
R_{i_{u}}  &  =r_{i+1}+i_{u}%
\end{align*}
Note that $R_{i_{\ell}},\ldots,R_{i}$ are consecutive integers. Likewise note
that $R_{i+1},\ldots,R_{i_{u}}$ are consecutive integers. We show that
$j_{u}-j_{\ell}=R_{i+1}-R_{i}$. Consider the following two cases:

\begin{description}
[leftmargin=4em,style=nextline,itemsep=0.5em]

\item[Case 1: $R_{i}+1=R_{i+1}$] Note that $R_{i_{\ell}},\ldots,R_{i_{u}}$ are
consecutive and they do not appear in $C.$ \ Since $C$ is sorted in increasing
order and $C_{j_{\ell}}=R_{i_{\ell}}-1\ $and $C_{j_{u}}=R_{i_{u}}+1$, there is
nothing in between $C_{_{j_{\ell}}}\ $and $C_{j_{u}}.$ Hence $j_{u}-j_{\ell
}=1.$ Note that $R_{i+1}-R_{i}=1.$ Thus $j_{u}-j_{\ell}=R_{i+1}-R_{i}.$

\item[Case 2: $R_{i}+1<R_{i+1}$] Note that the consecutive list of numbers
$R_{i}+1,\ldots,R_{i+1}-1$ do not appear in $R.$ Hence they appear in $C.$
Since $C$ is sorted in increasing order and $C_{j_{\ell}}=R_{i_{\ell}}-1\ $and
$C_{j_{u}}=R_{i_{u}}+1$, we conclude that exactly $R_{i}+1,\ldots,R_{i+1}-1$
appear in between $C_{j_{\ell}}$ and $C_{j_{u}}$. Hence
\[
j_{u}-j_{\ell}-1=\left(  R_{i+1}-1\right)  -\left(  R_{i}+1\right)
+1=R_{i+1}-R_{i}-1
\]
Thus $j_{u}-j_{\ell}=R_{i+1}-R_{i}.$
\end{description}

In both cases, we have shown that $j_{u}-j_{\ell}=R_{i+1}-R_{i}$. Note
\begin{align*}
c_{j_{u}}-c_{j_{\ell}}  &  =\left(  C_{j_{u}}-j_{u}\right)  -\left(
C_{j_{\ell}}-j_{\ell}\right) \\
&  =\left(  C_{j_{u}}-C_{j_{\ell}}\right)  -\left(  j_{u}-j_{\ell}\right) \\
&  =\left(  \left(  R_{i_{u}}+1\right)  -\left(  R_{i_{\ell}}-1\right)
\right)  -\left(  R_{i+1}-R_{i}\right) \\
&  =\left(  \left(  r_{i+1}+i_{u}+1\right)  -\left(  r_{i}+i_{\ell}-1\right)
\right)  -\left(  \left(  r_{i+1}+i+1\right)  -\left(  r_{i}+i\right)  \right)
\\
&  =i_{u}-i_{\ell}+1\\
&  \geq2
\end{align*}

Hence there exists $i\in\{1,\ldots,m\}$ such that $\mathcal{S}_{1,i,j_{\ell}%
}^{\prime}=0$ and $\mathcal{S}_{1,i,j_{u}}^{\prime}=1.$
\end{description}
\end{proof}

\begin{algorithm}
[$ResFromSyl$]\label{alg:ResFromSyl}\ 

\begin{description}
[leftmargin=3em,style=nextline,itemsep=0.0em]

\item[\textsf{In:}] $\left(  \mathcal{S}_{1},\mathcal{S}_{2}\right)  $, a
syl-representation of a term $t$

\item[\textsf{Out:}] $\mathcal{R}$, a res-representation of the term $t$
\end{description}

\begin{enumerate}
\item \label{Sort1} $\left(  \mathcal{S}_{1}^{\prime},\mathcal{S}_{2}^{\prime
}\right)  \leftarrow Sort\left(  \mathcal{S}_{1},\mathcal{S}_{2}\right)  $

\item \label{step4} $\left(  \mathcal{S}_{1}^{\prime},\mathcal{S}_{2}^{\prime
}\right)  \leftarrow Flush\left(  \mathcal{S}_{1}^{\prime},\mathcal{S}%
_{2}^{\prime}\right)  $

\item \label{res_rep} $\mathcal{R}\leftarrow\mathcal{S}^{\prime}_{1}$

\item Return $\mathcal{R}.$
\end{enumerate}
\end{algorithm}

\begin{example}
We trace the algorithm $ResFromSyl$ on the following input.

\begin{description}
[leftmargin=3em,style=nextline,itemsep=0.0em]

\item[\textsf{In:}] $\left(  \mathcal{S}_{1},\mathcal{S}_{2}\right)  =\left[
\begin{array}
[c]{cccc}%
0 & 1 & 1 & 1\\
0 & 1 & 1 & 1\\
1 & 1 & 1 & 0\\
0 & 0 & 1 & 1\\
0 & 1 & 1 & 0
\end{array}
\right]  , \left[
\begin{array}
[c]{cccc}%
0 & 0 & 0 & 0\\
0 & 1 & 0 & 0\\
1 & 0 & 1 & 0\\
0 & 0 & 0 & 0\\
1 & 1 & 1 & 1
\end{array}
\right]  $, \newline which is a syl-representation of the term $t=\alpha
_{1}^{3}\alpha_{2}^{3}\alpha_{3}^{3}\alpha_{4}^{2}\alpha_{5}^{2}\beta_{1}%
^{2}\beta_{2}^{2}\beta_{3}^{2}\beta_{4}^{1} $.
\end{description}

\begin{enumerate}
\item $\left(  \mathcal{S}_{1}^{\prime},\mathcal{S}_{2}^{\prime}\right)
=\left[
\begin{array}
[c]{cccc}%
0 & 1 & 1 & 1\\
0 & 1 & 1 & 1\\
1 & 1 & 0 & 1\\
0 & 0 & 1 & 1\\
0 & 1 & 1 & 0
\end{array}
\right]  , \left[
\begin{array}
[c]{cccc}%
0 & 0 & 0 & 0\\
0 & 1 & 0 & 0\\
0 & 0 & 1 & 0\\
1 & 0 & 0 & 0\\
1 & 1 & 1 & 1
\end{array}
\right]  $

\item $\left(  \mathcal{S}_{1}^{\prime},\mathcal{S}_{2}^{\prime}\right)
=\left[
\begin{array}
[c]{cccc}%
1 & 0 & 1 & 1\\
1 & 1 & 0 & 1\\
1 & 1 & 0 & 1\\
0 & 0 & 1 & 1\\
0 & 1 & 1 & 0
\end{array}
\right]  , \left[
\begin{array}
[c]{cccc}%
0 & 0 & 0 & 0\\
0 & 0 & 0 & 0\\
0 & 0 & 0 & 0\\
1 & 1 & 1 & 0\\
1 & 1 & 1 & 1
\end{array}
\right]  $

\item $\mathcal{R}=\left[
\begin{array}
[c]{cccc}%
1 & 0 & 1 & 1\\
1 & 1 & 0 & 1\\
1 & 1 & 0 & 1\\
0 & 0 & 1 & 1\\
0 & 1 & 1 & 0
\end{array}
\right]  $
\end{enumerate}

\begin{description}
\item[\textsf{Out:}] $\mathcal{R}=\left[
\begin{array}
[c]{cccc}%
1 & 0 & 1 & 1\\
1 & 1 & 0 & 1\\
1 & 1 & 0 & 1\\
0 & 0 & 1 & 1\\
0 & 1 & 1 & 0
\end{array}
\right]  $, \newline which is a res--representation of the term $t$.
\end{description}
\end{example}

\begin{lemma}
\label{ResFromSyl} The algorithm~\ref{alg:ResFromSyl} $(ResFromSyl)$ is
correct. Thus if a term has a syl-representation then it has a res-representation.
\end{lemma}

\begin{proof}
Let $\left(  \mathcal{S}_{1},\mathcal{S}_{2}\right)  $ be an input, that is, a
syl-representation of a term $t$. The correctness of the algorithm is
immediate from the following claims.

\begin{description}
[leftmargin=2.7em,style=nextline,itemsep=0.5em]

\item[\textsf{C1:}] \emph{After Step~\ref{Sort1}, $\left(  \mathcal{S}%
_{1}^{\prime},\mathcal{S}_{2}^{\prime}\right)  $ is a sorted syl
representation of the term $t$.}

Immediate from the specification of Algorithm~\ref{alg:Sort} ($Sort$) and
Lemma~\ref{lem:Sort}.

\item[\textsf{C2:}] \emph{After Step~\ref{step4}, $\left(  \mathcal{S}%
_{1}^{\prime},\mathcal{S}_{2}^{\prime}\right)  $ is a sorted-flushed
syl-representation of the term $t$. }

Immediate from \textsf{C1}, the specification of Algorithm~\ref{alg:Flush}
($Flush$) and Lemma~\ref{lem:Flush}.

\item[\textsf{C3:}] \emph{After Step~\ref{res_rep}, $\mathcal{R}$ is a res
representation of the term $t$}

Immediate from \textsf{C2} and Lemma~\ref{lem:Flushed}.
\end{description}
\end{proof}

\bigskip

\bibliographystyle{plain}
\def\cprime{$'$} \def\cprime{$'$} \def\cprime{$'$} \def\cprime{$'$}
  \def\cprime{$'$} \def\cprime{$'$} \def\cprime{$'$} \def\cprime{$'$}
  \def\cprime{$'$} \def\cprime{$'$} \def\cprime{$'$} \def\cprime{$'$}
  \def\cprime{$'$} \def\cprime{$'$} \def\cprime{$'$} \def\cprime{$'$}
  \def\cprime{$'$} \def\cprime{$'$} \def\cprime{$'$} \def\cprime{$'$}
  \def\cprime{$'$} \def\cprime{$'$} \def\cprime{$'$} \def\cprime{$'$}
  \def\cprime{$'$}

\end{document}